\documentclass[a4paper, 11pt]{article}
\usepackage{amssymb,amsmath,amsthm,lmodern,cancel}
\usepackage{amsfonts}

\usepackage[activate={true,nocompatibility},final,tracking=true,kerning=true,spacing=true,factor=1100,stretch=15,shrink=15]{microtype}
\microtypecontext{spacing=nonfrench}

\usepackage{graphicx}
\usepackage[english]{babel}
\usepackage[T1]{fontenc}
\usepackage{textcomp}
\usepackage[dvipsnames]{xcolor}
\usepackage[shortlabels]{enumitem}
\usepackage{thmtools, mathtools,verbatim} \usepackage{upgreek}
\usepackage{stackengine}
\usepackage[colorlinks=false,hyperindex=true,hypertexnames=false,hidelinks]{hyperref}
\usepackage{cleveref}
\usepackage[font=normal]{caption}
\usepackage{thmtools}
\usepackage{thm-restate}
\usepackage{tikz}

\title{Precompactness of sequences of random variables and random curves revisited}  
\date{\today}

\author{Osama Abuzaid\thanks{Department of Mathematics and Systems Analysis, Aalto University, Otakaari 1, FI-02150, Espoo, Finland. \protect\url{osama.abuzaid@aalto.fi}}}

\setlist[enumerate]{topsep = 1ex, leftmargin=1cm, itemsep= -2pt}

\let\OLDthebibliography\thebibliography
\renewcommand\thebibliography[1]{
	\OLDthebibliography{#1}
	\setlength{\parskip}{1pt}
	\setlength{\itemsep}{2pt}
}

\allowdisplaybreaks

\newtheorem{thm}{Theorem}[section]

\newtheorem{lem}[thm]{Lemma}
\newtheorem{prop}[thm]{Proposition}

\newtheorem{theorem}{Theorem}

\theoremstyle{definition} 
\newtheorem{df}[thm]{Definition}

\newtheorem{remark}[thm]{Remark}

\numberwithin{equation}{section}

\global\long\def\bR{\mathbb{R}}

\global\long\def\bN{\mathbb{N}}
\global\long\def\bQ{\mathbb{Q}}

\global\long\def\bP{\mathbb{P}}

\global\long\def\cA{\mathcal{A}}
\global\long\def\cH{\mathcal{H}}
\global\long\def\cS{\mathcal{S}}
\global\long\def\cP{\mathcal{P}}
\global\long\def\cF{\mathcal{F}}
\global\long\def\cX{\mathcal{X}}
\global\long\def\cL{\mathcal{L}}

\global\long\def\cM{\mathcal{M}}

\global\long\def\ud{\mathrm{d}}

\def\diam{{\rm diam}}
\newcommand{\dist}[1][]{{\rm dist}_{#1}}

\newcommand{\cond}{\;|\;}

\newcommand{\pathset}{\cP}
\newcommand{\PATHSET}{\mathbf{P}}
\newcommand{\Pathset}{\mathfrak{P}}
\newcommand{\compacts}{\mathcal K}
\newcommand{\dpaths}{d_{\pathset}}
\newcommand{\dPaths}{d_{\Pathset}}
\newcommand{\dHausdorff}{d_{\compacts}}

\stackMath
\newcommand\tsup[2][2]{%
	\def\useanchorwidth{T}%
	\ifnum#1>1%
	\stackon[-1.25ex]{\tsup[\numexpr#1-1\relax]{#2}}{\mathchar"307E\kern-.5pt}%
	\else%
	\stackon[-1ex]{#2}{\mathchar"307E\kern-.5pt}%
	\fi%
}
\newcommand{\ttilde}[1]{\tsup[2]{#1}}

\newcommand{\annulus}[3]{A_{#1,#2}(#3)}
\newcommand{\crossings}[4]{{\mathcal A}^{#4}_{#1,#2}(#3)}
\newcommand{\ncrossings}[4]{{\mathcal N}^{#4}_{#1,#2}(#3)}
\newcommand{\scrossings}[5]{{\mathcal A}^{#4}_{#1,#2}(#3;#5)}
\newcommand{\nscrossings}[5]{{\mathcal N}^{#4}_{#1,#2}(#3;#5)}
\newcommand{\restricted}[2]{{#1}|_{#2}}
\newcommand{\conditioned}[2]{{#1}^{\#}_{#2}}

\newcommand{\one}{\boldsymbol{1}}
\newcommand{\algsigma}{\Sigma}
\newcommand{\setring}{{\mathcal R}}

\newcommand{\tortuosity}{\mathcal{T}}
\newcommand{\compactX}{Y}
\newcommand{\regconst}{\lambda}
\newcommand{\bgamma}{\boldsymbol{p}}
\newcommand{\imag}{\operatorname{Im}}

\begin{document}
	\maketitle

	\begin{abstract}
		This paper studies when a sequence of probability measures on a metric space admits subsequential weak limits. A sufficient condition called approximate tightness is formulated, which relaxes some assumptions for asymptotic tightness used in the Prokhorov-Le Cam's theorem. The proof is relatively short and only uses elementary tools from probability theory. We also provide examples where approximate tightness is demonstrably easier to check than asymptotic tightness.
		
		Approximate tightness gives means to characterize precompact collections of random curves on a locally path-connected metric space \(\cX\) in terms of annulus crossing probability estimates which may be non-uniform in the modulus of annuli. This vastly generalizes the scope of the precompactness result by Aizenman and Burchard, which assumes \(\cX\) to be a compact subset of \(\bR^d\).
	\end{abstract}

	\tableofcontents

	\newpage
	\section{Introduction}\label{sec:intro}
	Given a sequence of probability measures \((\mu_n)_{n \in \bN}\) on a topological space \(\cX\) one is often interested in the existence of a possible limiting measure \(\mu\). A usual strategy is to first establish precompactness of the sequence \((\mu_n)_{n \in \bN}\), and then identify every possible subsequential weak limit of \((\mu_n)_{n \in \bN}\) to be equal to the same probability measure \(\mu\). In the present article we address the question of identifying precompact sequences of probability measures on metrizable topological spaces.
	
	Recall that a sequence of Borel measures \((\mu_n)_{n \in \bN}\) on \(\cX\) is said to converge weakly to a Borel measure \(\mu\) on \(\cX\) if for every bounded continuous function \(f : \cX \to \bR\) we have \(\lim_{n \to \infty}\int_\cX f \ud\mu_n = \int_\cX f \ud\mu\); this is denoted as \(\mu_n \xrightarrow{w} \mu\). We denote by \(\cM_1(\cX)\) the set of Borel probability measures on \(\cX\) and equip it with a topology of weak convergence (see e.g. the beginning of~\cite[Section 6]{billingsley1999convergence}). If the topology on \(\cX\) is metrizable, Prokhorov-Le Cam theorem \cite{prokhorov1956convergence,lecam1957convergence} provides a sufficient condition for precompactness called asymptotic tightness:
	\begin{theorem}[{E.g. \cite[Chapter 7, Theorem <36>]{pollard2002user}}]\label{thm:Prohorov-LeCam}
		Let \(\cX\) be a metrizable topological space. Suppose a sequence \((\mu_n)_{n \in \bN}\) in \(\cM_1(\cX)\) is asymptotially tight\footnote{\cite{pollard2002user} calls this uniform tightness.}: for every \(\varepsilon > 0\) there exists a compact set \(K_\varepsilon \subset \cX\) such that any open set \(G \subset \cX\) containing \(K_\varepsilon\) satisfies \(\liminf_{n \to \infty} \mu_n(G) \ge 1- \varepsilon\). Then \((\mu_n)_{n \in \bN}\) is a precompact sequence in \(\cM_1(\cX)\).
	\end{theorem}
	The compact set \(K_\varepsilon\) in asymptotic tightness has to be uniform over all of its open neighborhoods \(G \supset K_\varepsilon\). If \(\cX\) is equipped with a concrete complete metric \(d\), then the uniformity can be relaxed. For \(A \subset \cX\) and \(\delta > 0\), we will write
	\begin{align*}
		B_A(\delta) := \bigcup_{x \in A}B_x(\delta), \qquad \text{where} \qquad B_x(\delta) = \{y \in \cX \cond d(x, y) < \delta\}.
	\end{align*}
	\begin{theorem}[Prokhorov's theorem, e.g. {\cite[Theorem 2.2]{ethier2009markov}}]\label{thm:seq-tightness-complete}\label{thm:prokhorov}
		Let \((\cX, d)\) be a complete separable metric space, and \((\mu_n)_{n \in \bN}\) a sequence in \(\cM_1(\cX)\). Then the following are equivalent.
		\begin{enumerate}[label=(\roman*)]
			\item\label{item:prokhorov-tightness} The collection \(\{\mu_n\}_{n \in \bN}\) is a tight subset of \(\cM_1(\cX)\): for every \(\varepsilon > 0\) there exists a compact set \(K_\varepsilon \subset \cX\) such that \(\inf_{n \in \bN} \mu_n(K_\varepsilon) \ge 1-\varepsilon\).
			\item\label{item:prokhorov-approximate-tightness} For each \(\varepsilon, \delta > 0\) there exists a compact set \(K^\delta_\varepsilon \subset \cX\) such that \(\liminf_{n \to \infty}\mu_n\big(B_{K^\delta_\varepsilon}(\delta)\big) \ge 1 - \varepsilon.\)
			\item\label{item:prokhorov-precompact} \((\mu_n)_{n \in \bN}\) is a precompact sequence in \(\cM_1(\cX)\).
		\end{enumerate}
	\end{theorem}
	Inspecting the above two statements, a natural question arises: if \(\cX\) is not a complete space, could we still use the condition~\ref{item:prokhorov-approximate-tightness} from~\Cref{thm:prokhorov} if the compact sets \(K_\varepsilon^\delta\) for \(\delta > 0\) lay in a common complete subspace \(K_\varepsilon \subset \cX\)? Below we introduce the novel concept of approximate tightness, which formalizes this idea.
	\begin{df}[Approximate tightness]\label{def:approximate-tightness}
		Let \((\cX, d)\) be a metric space. A sequence \((\mu_n)_{n \in \bN}\) in \(\cM_1(\cX)\) is said to be \emph{approximately tight} if for each \(\varepsilon > 0\) there exists a complete subspace \(K_\varepsilon \subset \cX\) and a collection of compact subsets \((K_\varepsilon^\delta)_{\delta > 0}\) of \(K_\varepsilon\) such that
		\begin{align}\label{eqn:approximate-tightness}
			\liminf_{n \to \infty}\mu_n\big(B_{K^\delta_\varepsilon}(\delta)\big) \ge 1 - \varepsilon, \qquad \forall \delta > 0
		\end{align}
		We then say that \((\mu_n)_{n \in \bN}\) is approximately tight along the sets \((K^\delta_\varepsilon)_{\varepsilon, \delta > 0}\).	
	\end{df}
	In practice, to check approximate tightness, it may be easier to first construct the compact sets \(K^{\delta_n}_\varepsilon\) for a decreasing sequence \(\delta_n \searrow 0\), and then find a complete \(K_\varepsilon\) containing each \(K^{\delta_n}_{\varepsilon}\). This suffices, since for \(\delta \in (\delta_{n+1}, \delta_n)\) we may choose \(K^\delta_\varepsilon = K^{\delta_n}_\varepsilon\).
		
	The main aspect of the first main result,~\Cref{thm:main-thm} below, is that approximate tightness implies precompactness. Additionally, similar to how in~\Cref{thm:seq-tightness-complete} the a priori weaker condition~\ref{item:prokhorov-approximate-tightness} turns out to be equivalent to tightness,~\Cref{thm:main-thm} sets asymptotic and approximate tightness on equal footing. Thus, one may view approximate tightness as a metric dependent refinement of asymptotic tightness.
	\begin{restatable}{thm}{precompactnessthm}\label{thm:main-thm}Let \((\cX, d)\) be a metric space, and \((\mu_n)_{n \in \bN}\) a sequence in \(\cM_1(\cX)\). Denote by \(\cL := \bigcap_{n \in \bN}\overline{\{\mu_k\}_{k \ge n}}\) the set of possible subsequential limits of \((\mu_n)_{n \in \bN}\); here \(\overline{\{\mu_k\}_{k \ge n}}\) is the closure of the set \(\{\mu_k \cond k \ge n\}\) in \(\cM_1(\cX)\). Then the following are equivalent:
		\begin{enumerate}[(\roman*)]
			\item\label{item:asymptotic-tightness} \((\mu_n)_{n \in \bN}\) is asymptotically tight.
			\item\label{item:approximate-tightness} \((\mu_n)_{n \in \bN}\) is approximately tight.
			\item\label{item:precompact+tight-limits} \((\mu_n)_{n \in \bN}\) is precompact and \(\cL\) is a tight subset of \(\cM_1(\cX)\).
		\end{enumerate}
	\end{restatable}
	The equivalence of asymptotic and approximate tightness may seem to be just a mathematical curiosity, putting the novelty of approximate tightness to question. However, in the proof of~\Cref{thm:precompact-iff-regular} --- the second main result of this paper stated later --- not only is approximate tightness demonstrably easier to check than asymptotic tightness, but also the implication \ref{item:approximate-tightness} \(\implies\) \ref{item:asymptotic-tightness} of~\Cref{thm:main-thm} plays a crucial role.
	
	The first step in the proof of~\Cref{thm:main-thm} is to establish~\Cref{thm:approximate-tightness-implies-precompactness}, which states that approximate tightness implies precompactness and partially
	resolves the implication~\ref{item:approximate-tightness} \(\implies\)~\ref{item:precompact+tight-limits}. The proof of~\Cref{thm:approximate-tightness-implies-precompactness} is rather short and uses only elementary tools of probability theory. This is in contrast to standard textbook proofs of the Prokhorov-Le Cam's~\Cref{thm:Prohorov-LeCam} and the slightly weaker Prokhorov's \Cref{thm:prokhorov} theorems, which either use the Riesz-Markov-Kakutani representation theorem \cite[Chapter 7, Section 5]{pollard2002user} or outer measures and Carath\'eodory extension theorem \cite[Section 5]{billingsley1999convergence}. Under the simplifying assumption of separability of \(\cX\) Prokhorov's theorem may be proven with the help of L\'evy-Prokhorov metric (\cite[Chapter 3, Sections 1 and 2]{ethier2009markov} or~\cite[Section 6]{billingsley1999convergence}), which is elementary but long. Alternatively in the separable case, \cite{chin2019new} uses an embedding of \(\cX\) to the space \(\bR^\bN\) where Prokhorov's theorem is an easy consequence of Kolmogorov's extension theorem, whose proof on \(\bR^\bN\) is also rather short. 
	
	A key technical part in the proof of~\Cref{thm:approximate-tightness-implies-precompactness} is a construction of consistent coupling (see \Cref{def:consistent-coupling} and~\Cref{thm:consistent-couplings}) of a subsequence \((\mu_{n_k})_{k \in \bN}\). \Cref{thm:approximate-tightness-implies-precompactness} then follows almost directly once existence of consistent couplings is established. On the one hand, the coupling needed to prove that L\'evy-Prokhorov metric metrizes \(\cM_1(\cX)\) for complete separable metric spaces \((\cX, d)\) can be viewed as an approximate version of consistent coupling. On the other hand, existence of consistent coupling can be viewed as a consequence of Skorokhod's representation theorem (see e.g. \cite[Theorem 6.7]{billingsley1999convergence}). However, converse to Skorokhod's theorem which starts with a weakly converging sequence, we use consistent couplings to find weakly converging subsequences and thus cannot rely on Skorokhod's representation theorem.
	
	Interestingly, the author of the present article has not found a direct way to relate approximate tightness of a sequence \((\mu_n)_{n \in \bN}\) to asymptotic tightness in general. Recall that if \((K_n)_{n \in \bN}\) are compact subsets of a common complete metric space, and \((\delta_n)_{n \in \bN}\) are positive numbers such that \(\sum_{n \in \bN}\delta_n < \infty\), then the set \(K := \bigcap_{n \in \bN}\overline{B_{K_n}(\delta_n)}\) is compact. If \(\cX\) is complete (as in~\Cref{thm:prokhorov}), we may choose \(K_n = K^{\delta_n}_{2^{-n}\varepsilon}\) to get \(\liminf_{n \to \infty}\mu_n(K) \ge 1-\varepsilon\), which implies asymptotic tightness. However, for non-complete \(\cX\) the sets \(K^{\delta_n}_{2^{-n}\varepsilon}\) lay in different complete subspaces \(K_{2^{-n}\varepsilon}\) of \(\cX\), thus compactness of \(K\) fails in general. One could try to embed \(\cX\) to its completion \(\overline \cX\) to circumvent this issue, but \(\cX\) may not be a Borel subset of \(\overline \cX\), in which case \(\mu_n\) may not be well defined as a measure on \(\overline\cX\). It is thus not clear whether one could reduce the proof of~\Cref{thm:main-thm} to~\Cref{thm:Prohorov-LeCam} by directly proving the implication (ii) \(\implies\) (i) in~\Cref{thm:main-thm} without relying on~\Cref{thm:approximate-tightness-implies-precompactness}.
	
	We use \Cref{thm:main-thm} to characterize precompact sequences of certain random collections of curves. This is important for example when considering scaling limits of lattice interfaces in statistical physics~\cite{smirnov2001critical,kemppainen2017random,benoist2019scaling}. Precompactness for e.g. self-avoiding walks and double dimer interfaces is a long-standing open problem.
	
	Given a metric space \((\cX, d)\), denote by \(\pathset = \pathset(\cX)\) the set of unparameterized compact curves on \(\cX\) equipped with the uniform metric. Denote by \(\Pathset = \Pathset(\cX)\) the set of countable path collections \(\Gamma \subset \pathset\) containing only finitely many paths of diameter \(>\delta\) for each \(\delta > 0\), equipped with a metric analogous to the one in~\cite{benoist2019scaling}; for precise definitions, see the beginning of~\Cref{sec:precompact-curves}. Note that \(\Gamma \in \Pathset\) may still contain infinitely many small paths, hence e.g. conformal loop ensembles \cite{sheffield2009exploration} and Brownian loop soups \cite{lawler2004brownian} are almost surely elements of \(\Pathset\).\footnote{Conformal loop ensembles and Brownian loop soups are defined on the space of \emph{unrooted} loops, so technically they are not elements of \(\Pathset\) but rather the set of unrooted loop collections as defined in~\cite{benoist2019scaling}. The results and proofs of this paper also work in that setting.}
	
	For \(x \in \cX\) and \(R>r>0\), denote by \(\annulus{r}{R}{x} := \{y \in \cX : r<d(x,y)<R\}\) the open annulus centered at \(x\) with the inner and outer radii \(r\) and \(R\) respectively. For a collection of paths \(\Gamma \in \Pathset\), denote by \(\ncrossings{r}{R}{x}{\Gamma}\) the total number of times a curve in \(\Gamma\) crosses the annulus \(\annulus{r}{R}{x}\). Note that a compact curve crosses each annulus only finitely many times. Since in addition every path collection \(\Gamma \subset \Pathset\) has only finitely many paths of diameter \(\ge R-r\), we deterministically have \(\ncrossings{r}{R}{x}{\Gamma} < \infty\). This motivates the following definition of regularity, which may be interpreted as a probabilistic and unparameterized version of equicontinuity. Throughout this paper we use the following shorthand notation for \((\PATHSET, \bgamma) \in \{(\pathset, \gamma), (\Pathset, \Gamma)\}\), a measure \(\mu \in \cM_1(\PATHSET)\), and a predicate \(p : \PATHSET \to \{\texttt{true}, \texttt{false}\}\):
		\begin{align*}
			\mu[p(\bgamma)] := \mu\big(\{\bgamma \in \PATHSET \cond p(\bgamma) = \texttt{true}\}\big).
	\end{align*}
	\begin{df}\label{def:regularity}
		Let \(\PATHSET \in (\pathset, \Pathset)\). A subset \(M \subset \cM_1(\PATHSET)\) is regular at \(x \in \cX\) if for every \(R>r>0\) we have
		\begin{align}\label{eqn:path-precompactness}
			\lim_{N \to \infty}\sup_{\mu \in M}\mu[\ncrossings{r}{R}{x}{\bgamma} \ge N] = 0.
		\end{align}
		\(M\) is weakly regular at \(x\) if for every \(R > 0\) there exists \(r \in (0, R)\) such that~\eqref{eqn:path-precompactness} holds. The set \(M\) is (weakly) regular if it is (weakly) regular at every point \(x \in \cX\).
	\end{df}
	Let us demonstrate how regularity of \(M \subset \cM_1(\cP)\) implies precompactness when \(\cX\) is a compact geodesic metric space. Fix \(\varepsilon > 0\). Given \(u > 0\) we may fix a finite set \(F_u \subset \cX\) such that \(\bigcup_{x \in F_u}B_x(u) = \cX\). Regularity ensures that for sufficiently large \(N \in \bN\) the event \(\bigcap_{x \in F_u}\ncrossings{u}{2u}{x}{\gamma} \le N\) has probability \(\ge 1-\varepsilon\) for measures in \(M\). Under this event we may approximate the curve \(\gamma\) with a curve consisting of at most \(N|F_u|\) geodesic pieces between points in \(F_u\). There are at most \(|F_u|^{N|F_u|}\) such piecewise geodesic curves, thus choosing small enough \(u\) and large enough \(N\) we find finite sets \(K^\delta_\varepsilon \subset \pathset\) satisfying the probability bound \(\mu(B_{K_\varepsilon^\delta}(\delta)) \ge 1-\varepsilon\) for every \(\mu \in M\). The implication~\ref{item:prokhorov-approximate-tightness} \(\implies\)~\ref{item:prokhorov-precompact} in~\Cref{thm:prokhorov} then gives precompactness of \(M\). This rather simple strategy is generalized in~\Cref{sec:proof-of-regular-iff-precompact} to characterize precompactness sequences of random curves whose images are approximately tight, the second main result of this paper. Let us already emphasize that the proof of the following theorem relies heavily on the full strength of~\Cref{thm:main-thm}.
	\begin{restatable}{thm}{precompactiffregular}\label{thm:precompact-iff-regular}
		Let \(\PATHSET \in \{\pathset, \Pathset\}\). Suppose \((\cX, d)\) is a locally path-connected metric space, and take \(M \subset \cM_1(\PATHSET)\). Suppose for each \(\varepsilon > 0\) there exists a complete subspace \(\compactX_\varepsilon \subset \cX\) and compact subsets \((\compactX_\varepsilon^\delta)_{\delta > 0}\) of \(\compactX_\varepsilon\) such that for every \(\varepsilon, \delta > 0\) there is a precompact set \(M_\varepsilon^\delta \subset M\) such that
			\begin{align}\label{eqn:paths-asymptotic-tightness}
				\mu\big(\PATHSET(B_{\compactX_\varepsilon^\delta}(\delta))\big) \ge 1-\varepsilon, \qquad \forall \mu \in M\setminus M_\varepsilon^\delta.
			\end{align}
			Then the following are equivalent.
			\begin{enumerate}[label=(\roman*)]
				\item \(M\) is regular. \label{item:Mreg}
				\item \(M\) is weakly regular. \label{item:Mweakreg}
				\item \(M\) is precompact. \label{item:Mprecompact}
			\end{enumerate}
	\end{restatable}
	In particular, for compact \(\cX\) the condition~\eqref{eqn:paths-asymptotic-tightness} is automatically satisfied by choosing \(Y_\varepsilon^\delta = Y_\varepsilon = \cX\), thus (weak) regularity of \(M \subset \cM_1(\PATHSET)\) is strictly equivalent with precompactness.
	
	Note that (weak) regularity requires uniformity of the probability \(\mu[\ncrossings{r}{R}{x}{\bgamma} \ge N]\) only over \(N\), and not over the annuli \(\annulus{r}{R}{x}\). This is in contrast to earlier precompactness results in the literature \cite{aizenman1999holder,kemppainen2017random}, which require probability bounds uniform over all annuli of a given modulus. For example, on compact subsets of \(\bR^d\), \cite[Theorem 1.2]{aizenman1999holder} requires power bounds
	\begin{align}\label{eqn:Aizenman-Burchard-condition}
		\mu\Big[\ncrossings{r}{R}{x}{\gamma} \ge N\Big] \le K_N\Big(\frac{r}{R}\Big)^{\lambda_N}, \qquad \forall \mu \in M,
	\end{align}
	with uniform constants \(K_N, \lambda_N \ge 0\) depending only on \(N \in \bN\), where \(\lim_{N \to \infty}\lambda_N = \infty\). With these assumptions, \cite{aizenman1999holder} proves tightness of \(M\) by constructing equicontinuous parametrizations on sets \(K_\varepsilon \subset \pathset\) having uniformly large probability \(\mu(K_\varepsilon) \ge 1-\varepsilon\), and invoking Arzel\'a-Ascoli's theorem to show compactness of \(K_\varepsilon\). Precompactness of \(M\) then follows by Prokhorov's~\Cref{thm:prokhorov}.
	
	As discussed in~\Cref{rem:tortuosity-bounds}, for locally compact \(\cX\) \Cref{thm:precompact-iff-regular} can be proven using a similar strategy, but otherwise it is not clear how to directly obtain (asymptotic) tightness even if \(\cX\) is complete and the sets \(\compactX_\varepsilon\) can be chosen to be compact. Thus,~\Cref{thm:precompact-iff-regular}, and the weaker Propositions~\ref{thm:regular-paths} and~\ref{thm:regular-Paths}, serve as examples of statements for which approximate tightness is easier to verify. In~\Cref{sec:example} we use~\Cref{thm:precompact-iff-regular} to provide concrete examples of precompact sequences of random curves in a Hilbert space (\Cref{thm:example}), which further demonstrates usefulness of~\Cref{thm:main-thm} over the Prokhorov-Le Cam's~\Cref{thm:Prohorov-LeCam}. 
	
	On \(\bR^d\) we can refine the regularity condition further as follows. Suppose \(M \subset \cM_1(\PATHSET(\bR^d))\) is not regular. Whenever a curve crosses \(\annulus{r}{R}{x}\), it also passes through the set \(\partial B_x(\frac{R+r}{2})\). As the number of balls of diameter \(\varepsilon\) needed to cover the set \(\partial B_x(\frac{R+r}{2})\) is proportional to \(\varepsilon^{1-d}\), one thus expects the regularity of \(M\) at one of these balls to fail with probability proportional to \(\varepsilon^{d-1}\). Contraposing this heuristic yields the following result:
	\begin{restatable}{thm}{regulareuclideanthm}\label{thm:regular-curves-euclidean}
		Let \(\PATHSET \in \{\pathset, \Pathset\}\). A subset \(M \subset \cM_1(\PATHSET(\bR^d))\) is regular if and only if for every \(x \in \bR^d\) and \(R > 0\) we have
		\begin{align}\label{eqn:regularity-rate-bound}
			\lim_{N \to \infty}\sup_{\mu \in M}\mu[\ncrossings{r}{R}{x}{\bgamma} \ge N] = o(r^{d-1}), \qquad \text{as } r \to 0.
		\end{align}
	\end{restatable}
	Perhaps somewhat surprisingly, \Cref{thm:regular-curves-euclidean} implies that in the precompactness condition~\eqref{eqn:Aizenman-Burchard-condition} it is sufficient that \(\lambda_N > d-1\) for some \(N \in \bN\) uniform over \(n \in \bN\) (but not necessarily over \(x \in \cX\) or \(R > 0\)). It turns out that the above heuristic alone is not quite enough to derive \Cref{thm:regular-curves-euclidean}, and instead proves a slightly weaker statement, \Cref{thm:ball-version}, where the rate of convergence in~\eqref{eqn:regularity-rate-bound} is replaced by \(o(g(r)r^{d-1})\) for a non-negative function \(g\) satisfying \(\lim_{r \to 0+}g(r) = 0\). To get~\Cref{thm:regular-curves-euclidean} we need to use cubes in place of balls, which is done in the end of \Cref{sec:refinement-of-regularity-condition}.
	
	\subsection*{Acknowledgments}
	I am supported by the Academy of Finland grant number 340461 “Conformal invariance in planar random geometry,” and by the Academy of Finland Centre of Excellence Programme grant number 346315 “Finnish centre of excellence in Randomness and STructures (FiRST).” Preprint of the present article has appeared in~\cite{abuzaid2025precompactness}.
	
	I want to thank anonymous referees for their careful reading of this work and insightful questions and suggestions; responding to them made me improve the presentation and results of this paper significantly. I want to thank Luis Brummet for carefully going through the first proof of~\Cref{thm:main-thm} together with me, and Eveliina Peltola for valuable feedback of the drafts of this paper. I want to thank Kalle Kyt\"ol\"a for pointing out the similarity of the coupling constructed in~\Cref{thm:consistent-couplings} to Skorokhod's representation theorem, and Jonas T\"olle for letting me know of the version~\Cref{thm:Prohorov-LeCam} of Prohorov's theorem with \emph{asymptotic} tightness as a sufficient condition. I want to thank Alex Karrila for interesting discussions related to the results and proofs in this paper, and Aapo Pajala, Xavier Poncini and Liam Hughes who helped with the presentation of~\Cref{sec:refinement-of-regularity-condition}. I want to thank Liam Hughes also for the fruitful discussions about a possible counterexample of an approximately tight sequence of probability measures which is not asymptotically tight, which ultimately lead to the negative answer. I thank Victor Riviero for questions that improved the presentation of the article and for the reference to the equivalence of Prokhorov's theorem to~\Cref{thm:seq-tightness-complete}\ref{item:prokhorov-approximate-tightness} which I was not familiar with before.
	
	\section{Approximate tightness implies precompactness}\label{sec:approximate-tightness-implies-precompactness}
	In this section we prove~\Cref{thm:main-thm}, the first main result of the present article. As a starting point, we isolate the implication from approximate tightness to precompactness into its own proposition
	\begin{prop}\label{thm:approximate-tightness-implies-precompactness}
		Let \((\cX, d)\) be a metric space, and \((\mu_n)_{n \in \bN}\) an approximately tight sequence in \(\cM_1(\cX)\). Then \((\mu_n)_{n \in \bN}\) is precompact sequence in \(\cM_1(\cX)\)
	\end{prop}
	The proof of~\Cref{thm:approximate-tightness-implies-precompactness} is motivated by the following heuristic argument. Suppose \((X_n)_{n \in \bN}\) is a sequence of \(\cX\)-valued random variables with laws \(X_n \sim \mu_n\) which are approximately tight with respect to compact sets \((K_\varepsilon^\delta)_{\varepsilon, \delta > 0}\). By compactness, for each \(\varepsilon, \delta > 0\) there exists a finite set \(F_\varepsilon^\delta \subset K_\varepsilon^\delta\) such that \(\{B_x(\delta)\}_{x \in F_\varepsilon^\delta}\) is an open covering of \(K_\varepsilon^\delta\). Then \((B_x(2\delta))_{x \in F_\varepsilon^\delta}\) is an open covering of \(B_{K_\varepsilon^\delta}(\delta)\) thus, by approximate tightness, with probability \(1-\varepsilon\) there exists a subsequence \((X_{n_j})_{j \in \bN}\) which eventually stays in \(B_{x_\varepsilon^\delta}(2\delta)\) for some (random) \(x_\varepsilon^\delta \in F_\varepsilon^\delta \subset K_\varepsilon\). After diagonal extraction over \(\delta \in \bQ_+\), \((X_{n_j})_{j \in \bN}\) thus becomes a Cauchy sequence converging to the point \(\lim_{\delta \to 0}x_\varepsilon^\delta \in K_\varepsilon\), where the limit exists by completeness of \(K_\varepsilon\). All this happens with probability \(1-\varepsilon\), so taking \(\varepsilon \to 0\) yields almost sure existence of a subsequence along which \((X_{n_j})_{j \in \bN}\) converges to a point in \(\bigcup_{\varepsilon > 0}K_\varepsilon \subset \cX\).
	
	Note that without specifying the coupling of the random variables \((X_n)_{n \in \bN}\) there is no hope for the converging subsequence \((X_{n_j})_{j \in \bN}\) to be independent of the realization of \((X_n)_{n \in \bN}\), which is required for us to be able to take the weak limit along the subsequence. To carry out the heuristic argument precisely, we need a coupling which maximizes the probability that \(X_n\) and \(X_m\) are contained in the same set \(B_x(2\delta)\) for sufficiently large \(n, m \in \bN\), at least along a fixed subsequence. The following definition formalizes this idea.
	\begin{df}[Consistent couplings]\label{def:consistent-coupling}
		Let \((\cX, \Sigma)\) be a measurable space, and \(\cA \subset \Sigma\) a collection of measurable sets. A sequence \((X_n)_{n \in \bN}\) of \(\cX\)-valued random variables is said to be \(\cA\)-consistently coupled if for every \(A \in \cA\) the limit \(\lim_{n \to \infty}\one_A(X_n)\) exists almost surely.
	\end{df}
	For consistent couplings to exist, at least the limits \(\lim_{n \to \infty}\mu_n(A)\) should exist for every \(A \in \mathcal A\). If \(\mathcal A\) is a countable ring of sets, this turns out to be also sufficient as the following proposition states.
	\begin{prop}\label{thm:consistent-couplings}
		Let \((\mu_n)_{n \in \bN}\) be a sequence of probability measures on a measurable space \((\cX, \algsigma)\). Suppose \(\setring \subset \algsigma\) is a countable collection of measurable sets which is closed under finite unions, intersections and set differences. If \(\lim_{n \to \infty}\mu_n(R)\) exists for every \(R \in \setring\), then there exists an \(\setring\)-consistently coupled sequence \((X_n)_{n \in \bN}\) of \(\cX\)-valued random variables such that \(X_n\) has law \(\mu_n\). Furthermore, \(\bP[\lim_{n \to \infty}\one_R(X_n) = 1] = \lim_{n \to \infty}\mu_n(R)\) for every \(R \in \setring\).
	\end{prop}
	\begin{proof}
		Fix an enumeration \(\setring = \{R_k\}_{k \in \bN}\), and define the increasingly refined partitions \(\cS_k\) of \(\cX\) inductively on \(k \in \bN\) as 
		\begin{align*}
			\cS_{0} = \{\mathcal X\}, \quad \cS_k = \bigg(\bigcup_{A \in \cS_{k-1}}\{A \cap R_k, A \setminus R_k\}\bigg).
		\end{align*}
		Since each element of \(\cS := \bigcup_{k \in \bN} \cS_k\) is constructed as finite intersections and complements of elements in the ring of sets \(\setring\), we have \(\cS \subset \setring\), therefore the limits \(p(A) := \lim_{n \to \infty}\mu_n(A)\) exist for every \(A \in \cS\).
		
		For each \(n \in \bN\) and \(A \in \cS_k\), define the intervals \(I_n(A) := [l_n(A), r_n(A)) \subset [0, 1)\) inductively on \(k\) as follows. First, set \(I_n(\cX) := [0, 1)\), and then
		\begin{align*}
			\begin{split}
				I_n(A \cap R_{k+1}) =\;& \big[l_n(A), l_n(A) + \mu_n(A \cap R_{k+1})\big),\\
				I_n(A \setminus R_{k+1}) =\;& \big[r_n(A)- \mu_n(A \setminus R_{k+1}), r_n(A)\big),
			\end{split}
			\quad\qquad \forall A \in \cS_{k}.
		\end{align*}
		Note that each interval \(I_n(A)\) has length \(\mu_n(A)\), and the collections \(\{I_n(A)\}_{A \in \cS_k}\) are increasingly refined partitions of \([0, 1)\) which satisfy
		\begin{align}\label{eqn:In-partitions}
			I_n(A) = \bigcup_{\substack{A' \in \cS_\ell \\ A' \subset A}} I_n(A'), \qquad \forall A \in \cS_k, k \le \ell.
		\end{align}
		Since the limit \(p(A) = \lim_{n \to \infty}\mu_n(A)\) exist for every \(A \in \cS\), a straightforward induction over \(k \in \bN\) shows that for every \(A \in \cS_k\) the limits \(l(A) := \lim_{n \to \infty}l_n(A)\) and \(r(A) := \lim_{n \to \infty}r_n(A)\) exist and satisfy \(r(A)-l(A) = p(A)\). We thus get
		\begin{align}\label{eqn:limit-interval}
			I(A) := (l(A), r(A)) \overset{(*)}{=} \bigcup_{n_0 > k} \bigcap_{n > n_0} I_n(A), \qquad \forall A \in \cS,
		\end{align}
		where the \((*)\)-equality holds up to Lebesgue measure \(0\).
		
		Let \(\xi\) be a uniformly distributed random variable on \((0, 1)\). The idea is to couple \(\xi\) to random variables \(X_n \sim \mu_n\) in a way that for every \(A \in \cS_n\), \(\xi \in I_n(A)\) is equivalent with \(X_n \in A\). By~\eqref{eqn:In-partitions} the equivalence also holds for every \(A \in \cS_k\) with \(k \le n\). By~\eqref{eqn:limit-interval}, for a fixed \(A \in \cS_k\) the event \(\xi \in I(A)\) is then equivalent with the sequence \((X_n)_{n \in \bN}\) eventually staying in \(A\), which implies \(\lim_{n \to \infty} \one_A(X_n) = 1\). Since \(\{I(A)\}_{A \in \cS_k}\) partition a subset of \([0, 1)\) having full Lebesgue measure, consistent coupling of \((X_n)_{n \in \bN}\) then follows easily.
		
		For each \(n \in \bN\), denote by \(A_n(\xi) \in \cS_n\) the unique set such that \(\xi \in I_n(A_n(\xi))\). Consider the \(\cX\)-valued random variable \(X_n := X_n|{A_n(\xi)}\), where \(X_n|A\) is an independent random variable with law \(\mu_n(\cdot \cond A)\). By the law of total probability, for any \(E \in \algsigma\) we then have
		\begin{align*}
			\bP[X_n \in E] =\;& \sum_{A \in \cS_n}\bP[(X_n|A) \in E \cond \xi \in I_n(A)]\bP[\xi \in I_n(A)] \\
			=\;& \sum_{A\in \cS_n}\mu_n(E \cond A)\mu_n(A) = \mu_n(E),
		\end{align*}
		therefore \(X_n \sim \mu_n\). Furthermore, for any \(k \le n\) and \(A \in \cS_k\) we have the equivalence of events
		\begin{align*}
			\{\xi \in I_n(A)\} \overset{\eqref{eqn:In-partitions}}{=} \bigcup_{\substack{A' \in \cS_n \\ A' \subset A}} \{\xi \in I_n(A')\} = \bigcup_{\substack{A' \in \cS_n \\ A' \subset A}} \{X_n \in A'\} = \{X_n \in A\},
		\end{align*}
		where the middle equality holds by definition of \(X_n = X_n|A_n(\xi)\). For each \(A \in \cS_k\), we thus get the following inclusion of events:
		\begin{gather*}
			\big\{\lim_{n \to \infty}\one_A(X_n) = 1\big\} =  \bigcup_{n_0 \ge k}\bigcap_{n \ge n_0}\big\{X_n \in A\big\} = \bigcup_{n_0 \ge k}\bigcap_{n \ge n_0}\big\{\xi \in I_n(A)\big\} = \big\{\xi \in I(A)\big\},\\
			\big\{\lim_{n \to \infty}\one_A(X_n) = 0\big\} \supset \bigcup_{B \in \cS_k\setminus\{A\}}\big\{\lim_{n \to \infty}\one_B(X_n) = 1\big\} = \bigcup_{B \in \cS_k \setminus \{A\}} \{\xi \in I(B)\}.
		\end{gather*}
		Since by~\eqref{eqn:limit-interval} we have \(\bP[\bigcup_{B \in \cS_k}\{\xi \in I(B)\}] = \sum_{B \in \cS_k}p(B) = 1\), we conclude almost sure existence of the limit \(\lim_{n \to \infty}\one_A(X_n)\). For each \(k \in \bN\) the set \(R_k \in \setring\) is partitioned by sets in \(\cS_k\), thus we get the almost sure equality
		\begin{align*}
			\lim_{n \to \infty}\one_{R_k}(X_n) = \lim_{n \to \infty}\sum_{\substack{A \in \cS_k \\ A \subset R_k}}\one_A(X_n) = \sum_{\substack{A \in \cS_k \\ A \subset R_k}} \lim_{n \to \infty}\one_A(X_n),
		\end{align*}
		so the limit \(\lim_{n \to \infty}\one_{R_k}(X_n)\) almost surely exists. Finally, since \(\cS_k\) consists of pairwise disjoint sets, at most one of \(\lim_{n \to \infty}\one_A(X_n)\) from the right hand side above can be non-zero for each realization of \((X_n)_{n \in \bN}\). We thus get 
		\begin{align*}
			\bP\big[\lim_{n \to \infty}\one_{R_k}(X_n) = 1\big] =\;& \sum_{\substack{A \in \cS_k \\ A \subset B_k}}\bP\big[ \lim_{n \to \infty}\one_A(X_n) = 1\big] \\
			=\;& \sum_{\substack{A \in \cS_k \\ A \subset R_k}} \bP[\xi \in I(A)] = \sum_{\substack{A \in \cS_k \\ A \subset R_k}} p(A) = \lim_{n \to \infty}\mu_n(R_k),
		\end{align*}
		where the last equality follows by changing the order of finite sum and the limits \(p_A = \lim_{n \to \infty}\mu_n(A)\). This concludes the proof.
	\end{proof}
	\begin{remark}
		Using inclusion exclusion principle one can show that in~\Cref{thm:consistent-couplings} it suffices to only assume that the limit \(\lim_{n \to \infty}\mu_n(B)\) exists for sets in a pi-system \(\Pi \subset \setring\) which generates \(\setring\) (i.e. every element in \(\setring\) is a finite combination of unions, intersections and set differences of sets in \(\Pi\)).
	\end{remark}
	
		We also need the fact that Cauchy sequences near a complete subspace converge.
		\begin{lem}\label{thm:Cauchy-near-complete-converges}
			Let \((\cX, d)\) be a metric space, and \(K \subset \cX\) a complete subspace. If \((x_j)_{j \in \bN}\) is a Cauchy sequence in \(\cX\) satisfying \(\lim_{j \to \infty}\dist(x_j, K) = 0\), then the limit \(\lim_{j \to \infty} x_j\) exists and lies in \(K\).
		\end{lem}
		\begin{proof}
			For each \(j \in \bN\), pick an element \(y_j \in K\) satisfying \(d(x_j, y_j) \le \dist(x_j, K) + \frac{1}{j}\). Then \((y_j)_{j \in \bN}\) is a Cauchy sequence in the complete space \(K\), so the limit \(y := \lim_{j \to \infty} y_j \in K\) exists. Since \(\lim_{j \to \infty}d(x_j, y_j) = 0\), the limits of \((x_j)_{j \in \bN}\) and \((y_j)_{j \in \bN}\) coincide, so \(\lim_{j \to \infty} x_j = y\).
		\end{proof}
	Now we are ready to flesh out the heuristic argument from the beginning of this section into a proof.
	\begin{proof}[Proof of~\Cref{thm:approximate-tightness-implies-precompactness}]
		Since approximate tightness is preserved when taking subsequences, it suffices to show that \((\mu_n)_{n \in \bN}\) has a converging subsequence. Let \(K_\varepsilon\) and \(K_\varepsilon^\delta\) be as in the definition of approximate tightness  (\Cref{def:approximate-tightness}). By compactness, for each \(\varepsilon, \delta > 0\) there exists a finite set \(F_\varepsilon^\delta \subset K_\varepsilon^\delta\) such that \(K_\varepsilon^\delta \subset B_{F_\varepsilon^\delta}(\delta)\). Then \(B_{K_\varepsilon^\delta}(\delta) \subset B_{F_\varepsilon^\delta}(2\delta) = \bigcup_{x \in F_\varepsilon^\delta}B_x(2\delta) =: U_\varepsilon^\delta\), thus we have
		\begin{align*}
			\liminf_{n \to \infty}\mu_n(U_\varepsilon^\delta) = \liminf_{n \to \infty}\mu_n\Big(\bigcup_{x \in F_\varepsilon^\delta}B_x(2\delta)\Big) \ge \liminf_{n \to \infty}\mu_n(B_{K_\varepsilon^\delta}(\delta)) \ge 1-\varepsilon.
		\end{align*}
		Let \(\setring\) be the collection of sets generated by taking finite unions, intersections and differences of sets in the countable set collection \(\bigcup_{\varepsilon, \delta \in \bQ_+}\{B_x(2\delta) \cond x \in F_\varepsilon^\delta\}\). Then \(\setring\) is countable, thus by diagonal extraction we may pass to a subsequence \((\mu_{n_k})_{k \in \bN}\) where the limit \(\lim_{k \to \infty}\mu_{n_k}(R)\) exists for every \(R \in \mathcal \setring\). By~\Cref{thm:consistent-couplings}, we may fix an \(\setring\)-consistently coupled sequence of random variables \((X_k)_{k \in \bN}\) having laws \(X_k \sim \mu_{n_k}\). Since \(U_\varepsilon^\delta \in \setring\) for every \(\varepsilon, \delta \in \bQ_+\), by reverse monotone convergence theorem we get (for \(\varepsilon \in \bQ_+\))
		\begin{align*}
			\bP\bigg[\bigcap_{\delta \in \bQ_+}\bigcup_{\substack{\delta' \in \bQ_+ \\ \delta' \le \delta}}\big\{\lim_{k \to \infty}\one_{U_\varepsilon^{\delta'}}(X_k) = 1\big\}\bigg] \ge\;& \limsup_{\substack{\delta \to 0 \\ \delta \in \bQ_+}}\bP[\lim_{k \to \infty}\one_{U_\varepsilon^\delta}(X_k) = 1] \\
			=\;& \limsup_{\substack{\delta \to 0 \\ \delta \in \bQ_+}}\lim_{k \to \infty}\mu_{n_k}[U_\varepsilon^\delta] \\
			\ge\;& 1-\varepsilon.
		\end{align*}
		Since \(\lim_{k \to \infty}\one_{U^\delta_\varepsilon}(X_k) \le \sum_{x \in F_\varepsilon^\delta}\lim_{k \to \infty}\one_{B_x(2\delta)}(X_k)\), we conclude that with probability \(1-\varepsilon\) there exists a sequence \(\delta_j \searrow 0\) in \(\bQ_+\) and points \(x_j \in F_\varepsilon^{\delta_j}\) such that for any \(j \in \bN\) and sufficiently large \(k \in \bN\) we have \(X_k \in B_{x_j}(2\delta_j)\). Thus \((X_k)_{k \in \bN}\) is a Cauchy sequence approaching the complete set \(K_\varepsilon\), so by~\Cref{thm:Cauchy-near-complete-converges} the limit \(\lim_{k \to \infty}X_k \in K_\varepsilon\) exists. All this happens with probability \(1-\varepsilon\), thus we get
		\begin{align*}
			\bP[\lim_{k \to \infty} X_k \text{ exists}] \ge \bP[\lim_{k \to \infty} X_k \in K_\varepsilon \text{ exists} ] \ge 1-\varepsilon \xrightarrow{\varepsilon \to 0} 1.
		\end{align*}
		We conclude almost sure convergence \(\lim_{k \to \infty} X_k\), which implies weak convergence of their laws \((\mu_{n_k})_{k \in \bN}\) and finishes the proof.
	\end{proof}
	\subsection{Equivalence of approximate and asymptotic tightness}
	In this subsection we prove the full statement of~\Cref{thm:main-thm}. In addition to~\Cref{thm:approximate-tightness-implies-precompactness} we use the fact that conditioning on closed and separable subspace preserves precompactness (\Cref{thm:conditioning-preserves-precompactness} below). Given a probability measure \(\mu \in \cM_1(\cX)\) and a set \(A \subset \cX\) such that \(\mu(A) > 0\), we denote by \(\conditioned{\mu}{A} \in \cM_{1}(\cX)\) the conditioned measure
	\begin{align*}
		\conditioned{\mu}{A}(B) := \frac{\mu(A \cap B)}{\mu(A)} \qquad \forall B \subset \cX \text{ Borel}.
	\end{align*}
	\begin{restatable}{prop}{conditioning}\label{thm:conditioning-preserves-precompactness}
		Let \(\cX\) be a metric space. Suppose \(\mathcal C\) is a precompact subset of \(\cM_{1}(\cX)\), and \(Y \subset \cX\) is a separable closed subset such that \(\inf_{\mu \in \mathcal C}\mu(Y) > 0\). Then the conditioned measures
		\begin{align*}
			\conditioned{\mathcal C}{Y} := \{\conditioned{\mu}{Y} \cond \mu \in \mathcal C\} \subset \cM_{1}(\cX)
		\end{align*}
		form a precompact subset of \(\cM_1(\cX)\).
	\end{restatable}
	This is a rather elementary result, but the author of the present article has not found it from the literature. For completeness, we thus present a surprisingly lengthy and technical proof of~\Cref{thm:conditioning-preserves-precompactness} in~\Cref{sec:appendix}.
	\begin{proof}[Proof of~\Cref{thm:main-thm}]
		\ref{item:asymptotic-tightness} \(\implies\) \ref{item:approximate-tightness}: Suppose \((\mu_n)_{n \in \bN}\) is asymptotically tight. Then for every \(\varepsilon > 0\) there exists a compact set \(K_\varepsilon \subset \cX\) such that for every open neighborhood \(U \subset \cX\) of \(K_\varepsilon\) we have \(\liminf_{n \to \infty}\mu_n(U) \ge 1-\varepsilon\). Setting \(K^\delta_\varepsilon = K_\varepsilon\) for every \(\delta, \varepsilon > 0\) the sequence \((\mu_n)_{n \in \bN}\) is approximately tight along the sets \((K^\delta_\varepsilon)_{\varepsilon,\delta > 0}\). Indeed, \(K_\varepsilon = \overline{\bigcup_{\delta > 0} K^\delta_\varepsilon}\) is a compact, hence complete subspace of \(\cX\), while by asymptotic tightness we have \(\liminf_{n \to \infty}\mu_n(B_{K^\delta_\varepsilon}(\delta)) \ge 1-\varepsilon\).

		\ref{item:approximate-tightness} \(\implies\) \ref{item:precompact+tight-limits}: Assume \((\mu_n)_{n \in \bN}\) is approximately tight. \Cref{thm:approximate-tightness-implies-precompactness} gives precompactness of \((\mu_n)_{n \in \bN}\), so it remains to check tightness of \(\cL\). Since by precompactness \(\overline{\{\mu_k\}_{k \ge n}}\) is compact for every \(n \in \bN\), so is \(\cL = \bigcap_{n \in \bN}\overline{\{\mu_k\}_{k \ge n}}\). By definition of \(\cL\), for any \(\mu \in \cL\) there exists a subsequence \((\mu_{n_k})_{k \in \bN}\) converging weakly to \(\mu\). For \(\varepsilon > 0\), consider the set \(\tilde K_\varepsilon := \overline{\bigcup_{\delta \in \bQ_+}K_\varepsilon^\delta}\), which is separable as a closure of a countable union of compact (thus separable) subspaces. By Portmanteau's theorem (\Cref{thm:Portmanteau}), for every \(\delta \in \bQ_+\) we have
		\begin{align*}
			\mu\big(\overline{B_{\tilde K_\varepsilon}(\delta)}\big) \ge \mu\big(\overline{B_{K_\varepsilon^{\delta}}(\delta)}\big) \ge \liminf_{k \to \infty} \mu_{n_k}\big(\overline{B_{K_\varepsilon^{\delta}}(\delta)}\big) \ge 1-\varepsilon,
		\end{align*}
		where the last bound holds by approximate tightness. By monotone convergence theorem we thus get
		\begin{align*}
			\mu(\tilde{K_\varepsilon}) = \lim_{\delta \to 0}\mu\big(\overline{B_{\tilde K_\varepsilon}(\delta)}\big) \ge 1-\varepsilon.
		\end{align*} 
		Since \(\tilde K_\varepsilon\) is closed and separable, by~\Cref{thm:conditioning-preserves-precompactness} the conditioned measures \(\conditioned{\cL}{\tilde K_\varepsilon} := \{\conditioned{\mu}{\tilde K_\varepsilon} \cond \mu \in \cL\}\) form a precompact subset of \(\cM_1(\tilde K_\varepsilon)\). Furthermore, \(\tilde K_\varepsilon\) is also complete as a closed subspace of the complete space \(K_\varepsilon\), therefore Prokhorov's theorem (\Cref{thm:prokhorov} (iii) \(\implies\) (i)) implies that \(\conditioned{\cL}{\tilde K_\varepsilon}\) is tight. There thus exists a compact set \(C_\varepsilon \subset \tilde K_\varepsilon\) such that
		\begin{align*}
			\conditioned{\mu}{\tilde K_\varepsilon}(C_\varepsilon) \ge 1-\varepsilon \qquad \forall \mu \in \cL,
		\end{align*}
		which yields 
		\begin{align*}
			\mu(C_\varepsilon) = \mu(\tilde K_\varepsilon)\mu^\#_{\tilde K_\varepsilon}(C_\varepsilon) \ge (1-\varepsilon)^2 \qquad \forall \mu \in \cL,
		\end{align*}
		proving tightness of \(\cL\).
		
		\ref{item:precompact+tight-limits} \(\implies\) \ref{item:asymptotic-tightness}: Assume \((\mu_n)_{n \in \bN}\) is precompact and \(\cL := \bigcap_{n \in \bN}\overline{\{\mu_k\}_{k \ge n}}\) is tight. Then, for every \(\varepsilon > 0\) there exists a compact set \(K_\varepsilon \subset \cX\) such that \(\mu(K_\varepsilon) \ge 1-\varepsilon\) for every \(\mu \in \cL\). Let \(U \subset \cX\) be an open neighborhood of \(K_\varepsilon\), and \((\mu_{n_k})_{k \in \bN}\) a subsequence satisfying \(\lim_{k \to \infty}\mu_{n_k}(U) = \liminf_{n \to \infty}\mu_n(U)\). By precompactness of the sequence \((\mu_n)_{n \in \bN}\), passing to a further subsequence we may without loss of generality assume that \((\mu_{n_k})_{k \in \bN}\) converges weakly to some \(\mu \in \cL\). By Portmanteau's theorem (\Cref{thm:Portmanteau}, \ref{item:weak-convergence} \(\implies\) \ref{item:open-sets}), we thus get
		\begin{align*}
			\liminf_{n \to \infty}\mu_n(U) = \lim_{k \to \infty}\mu_{n_k}(U) \ge \mu(U) \ge \mu(K_\varepsilon) \ge 1-\varepsilon,
		\end{align*}
		proving asymptotic tightness of \((\mu_n)_{n \in \bN}\).
	\end{proof}
	\section{Regular sequences of random collections of curves are precompact}\label{sec:precompact-curves}
	In this section we prove~Theorems \ref{thm:precompact-iff-regular} and \ref{thm:regular-curves-euclidean}. Let us begin with properly defining the notations introduced in~\Cref{sec:intro}. 
	
	Given a metric space \((\cX, d)\), denote by \(\pathset = \pathset(\cX)\) the set of compact curves on \(\cX\) modulo reparametrization. More precisely, two nowhere constant continuous maps \(\tilde \gamma_1 : I_1 \to \cX\) and \(\tilde \gamma_2 : I_2 \to \cX\) with compact intervals \(I_1, I_2 \subset \bR\) represent the same element \(\gamma \in \pathset\) if there exists an increasing homeomorphism \(\sigma: I_1 \to I_2\) such that \(\tilde \gamma_1 = \tilde \gamma_2 \circ \sigma\). Equip \(\pathset\) with the uniform metric \(\dpaths\) on unparametrized curves:
	\begin{align*}
		\dpaths(\gamma, \eta) := \inf_{\tilde \gamma, \tilde \eta} \sup_{t \in I}d(\tilde \gamma(t), \tilde \eta(t), \qquad \forall \gamma, \eta \in \pathset
	\end{align*}
	where the infimum is taken over all representatives \(\tilde \gamma, \tilde \eta : I \to \cX\) of \(\gamma\) and \(\eta\). (See e.g. \cite[Lemma 2.1]{aizenman1999holder} for a proof that \(\dpaths\) is a metric.) If there is no risk of misunderstanding, we often abuse the notation and denote by \(\gamma\) both an element of \(\pathset\) and a fixed representative \(\tilde\gamma: I \to \cX\) of it.
	
	Denote by \(\Pathset = \Pathset(\cX)\) the set of countable path collections \(\Gamma \subset \pathset\) containing only finitely many macroscopic paths, and no trivial paths; more formally, for any \(\delta > 0\) the set
	\begin{align*}
		\Gamma(\delta) := \{\gamma \in \Gamma \cond \diam(\gamma) > \delta\}
	\end{align*}
	is finite, and \(\Gamma = \bigcup_{\delta > 0} \Gamma(\delta)\). We treat the collections \(\Gamma\) and \(\Gamma(\delta)\) as multisets, so they can contain the same element of \(\pathset\) multiple times. A matching between two path collections \(\Gamma_1, \Gamma_2 \in \Pathset\) is a set \(\pi \subset \Gamma_1 \times \Gamma_2\) such that the projections \(\pi \to \Gamma_1\) and \(\pi \to \Gamma_2\) are injective; this ensures that as a function, \(\pi\) maps the paired elements \((\gamma_1, \gamma_2) \in \pi\) bijectively to each other. We write \(\Gamma_j^\pi \subset \Gamma_j\), \(j \in \{1, 2\}\), for the elements not contained in any pair in \(\pi\), and equip \(\Pathset\) with the metric
	\begin{align*}
		\dPaths(\Gamma_1, \Gamma_2) := \inf_\pi \max\bigg(\sup_{(\gamma_1, \gamma_2) \in \pi} \dpaths(\gamma_1, \gamma_2), \sup_{\gamma_1 \in \Gamma_1^\pi}\diam(\gamma_1), \sup_{\gamma_2 \in \Gamma_2^\pi} \diam(\gamma_2)\bigg),
	\end{align*}
	where the infimum is over all matchings \(\pi \subset \Gamma_1 \times \Gamma_2\). A similar metric in the case of collections of unrooted loops appear in \cite{benoist2019scaling}. Positivity and symmetricity of \(\dPaths\) is clear, while the triangle inequality follows directly by composing two matchings: for elements \(\Gamma_1, \Gamma_2, \Gamma_3 \in \Pathset\) and matchings \(\pi_{12} \subset \Gamma_1 \times \Gamma_2\) and \(\pi_{23} \subset \Gamma_2 \times \Gamma_3\), the composed matching is defined as
		\begin{align*}
			\pi_{12}\pi_{23} := \{(\gamma_1, \gamma_3) \in \Gamma_1 \times \Gamma_3 \cond \exists \gamma_2 \in \Gamma_2 \text{ s.t. } (\gamma_1, \gamma_2) \in \pi_{12} \text{ and } (\gamma_2, \gamma_3) \in \pi_{23}\}.
		\end{align*}
	Note that any two elements of \(\Pathset\) which are close in the metric \(\dPaths\) have the same number of paths of large diameter. This is opposed to the coarser Hausdorff metric considered in~\cite[see eq. (1.14)]{aizenman1999holder}, in which any neighbourhood of \(\Gamma \in \Pathset\) contains elements of \(\Pathset\) with arbitrarily many paths near any \(\gamma \in \Gamma\).
	
	For a point \(x \in \cX\) and radii \(0 < r < R < \infty\), an \(\annulus{r}{R}{x}\)-crossing is a compact path \(\gamma : [a ,b] \to \cX\) such that
	\begin{align*}
		\gamma(a, b) \subset A_{r, R}(x), \quad \text{ and } \quad \{d(x, \gamma(a)), d(x, \gamma(b))\} = \{r, R\}.
	\end{align*}
	For \(\gamma \in \pathset\) and its representative \(\tilde \gamma : I \to \cX\), and a path collection \(\Gamma \in \Pathset\), we define
	\begin{gather*}
		\crossings{r}{R}{x}{\tilde\gamma} := \big\{(a, b) \subset I \cond \tilde\gamma|_{[a, b]} \text{ is an } \annulus{r}{R}{x} \text{ crossing}\big\},\\
		\ncrossings{r}{R}{x}{\gamma} := |\crossings{r}{R}{x}{\tilde\gamma}|, \qquad \ncrossings{r}{R}{x}{\Gamma} := \sum_{\gamma \in \Gamma}\ncrossings{r}{R}{x}{\gamma}.
	\end{gather*}
	The number \(\ncrossings{r}{R}{x}{\gamma}\) counts the number of times the path \(\gamma\) crosses the annulus \(\annulus{r}{R}{x}\) and hence is independent of the parametrization of \(\gamma\). Therefore, \(\ncrossings{r}{R}{x}{\Gamma}\) is well defined. 
	
	If the underlying metric space \(\cX\) is complete, then both \(\pathset\) and \(\Pathset\) are complete. This is a consequence of Lemmas~\ref{thm:paths-near-complete-are-complete} and~\ref{thm:pathsets-of-complete-paths-is-complete} below, former of which may be viewed as an extension of~\Cref{thm:Cauchy-near-complete-converges} to paths.
		\begin{lem}\label{thm:paths-near-complete-are-complete}
			Let \((\cX, d)\) be a metric space. Suppose \(C \subset \cX\) is complete subspace, and \(P \subset \pathset(\cX)\) is such that \(P \setminus \pathset(B_C(\delta))\) is finite for every \(\delta > 0\). Then \(P \cup \pathset(C)\) is a complete subspace of \(\pathset(\cX)\).
		\end{lem}
		Let us remark that if \(C\) is not path connected, \(\pathset(B_C(\delta))\) may contain paths with arbitrarily large distances to \(\pathset(C)\), e.g. if \(C \subset \bR^2\) is the closure of \(\{(x, \sin(x^{-1})) \cond x \in \bR\setminus\{0\}\}\). Therefore the above~\Cref{thm:paths-near-complete-are-complete} does not directly follow from completeness of \(\pathset(C)\) and~\Cref{thm:Cauchy-near-complete-converges}.
		\begin{proof}[Proof of~\Cref{thm:paths-near-complete-are-complete}]
			Since for every \(\delta > 0\) all but finitely many elements of \(P\) are in \(\pathset(B_C(\delta))\), to prove completeness of \(P \cup \pathset(B_C(\delta))\) it suffices to prove convergence of Cauchy sequences \((\gamma_n)_{n \in \bN}\) in \(P \cup \pathset(C)\) which eventually stay in \(\pathset(B_C(\delta))\) for every \(\delta > 0\). Without loss of generality, assume \(\dpaths(\gamma_n, \gamma_{n+1}) < 2^{-n}\). Then there exist representatives \(\tilde\gamma_n : [0, 1] \to \cX\) of \(\gamma_n\) such that
			\begin{align}
				\sup_{t \in [0, 1]}d(\tilde\gamma_{n+1}(t), \tilde\gamma_n(t)) \le 2^{-n} \label{eqn:cauchy}
			\end{align}
			For each \(t \in [0, 1]\) the sequence \((\tilde\gamma_n(t))_{n \in \bN}\) is thus Cauchy and satisfies \linebreak
			\(\lim_{n \to \infty}\dist(\tilde\gamma_n(t), C) = 0\). By~\Cref{thm:Cauchy-near-complete-converges} and equation~\eqref{eqn:cauchy} we conclude uniform convergence of \(\tilde\gamma_n\) to a function \(\tilde\gamma : [0, 1] \to C\), which by uniform limit theorem is continuous. Thus, \(\tilde\gamma\) represents an element \(\gamma \in \pathset(C)\). Finally, we have  \(\dpaths(\gamma_n, \gamma) \le \sup_{t \in [0, 1]}d(\tilde\gamma_n(t), \tilde\gamma(t)) \xrightarrow{n \to \infty} 0\), which proves the convergence \(\gamma_n \to \gamma\) in \(P \cup \pathset(C)\) and thus completeness of \(P \cup \pathset(C)\).
		\end{proof}
		The next result says that completeness of a subspace \(P \subset \pathset\) is transferred to the set of path collections \(\Gamma \in \Pathset\) where each path \(\gamma \in \Gamma\) is in \(P\). In particular, choosing \(P = \pathset(\cX)\) for a complete space \(\cX\) gives completeness of \(\Pathset(\cX)\).
		\begin{lem}\label{thm:pathsets-of-complete-paths-is-complete}
			Let \((\cX, d)\) be a metric space, and suppose \(P \subset \pathset(\cX)\) is complete. Then the subspace
			\begin{align*}
				\Pathset(P) := \{\Gamma \in \Pathset(\cX) \cond \gamma \in P\;\forall \gamma \in \Gamma\}
			\end{align*}
			of \(\Pathset(\cX)\) is complete.
		\end{lem}
		\begin{proof}
			Let \((\Gamma_n)_{n \in \bN}\) be a Cauchy sequence in \(\Pathset(P)\); without loss of generality, assume \(\dPaths(\Gamma_n, \Gamma_{n+1}) < 2^{-n}\) for every \(n \in \bN\). Then for each \(n \in \bN\) there exists a matching \(\pi_n \subset \Gamma_n \times \Gamma_{n+1}\) such that 
			\begin{align}\label{eqn:pin-distance-bound}
				\max\bigg(\sup_{(\gamma, \hat\gamma) \in \pi_n}\dpaths(\gamma, \hat\gamma), \sup_{\gamma \in \Gamma^{\pi_n}_n} \diam(\gamma), \sup_{\hat\gamma \in \Gamma^{\pi_n}_{n+1}}\diam(\hat\gamma)\bigg) < 2^{-n}.
			\end{align}
			Therefore, if we fix \(d_n := 4\cdot 2^{-n}\) and take \(\delta \ge d_n\) for \(n \in \bN\), then \(\pi_n\) that matches every element \(\gamma \in \Gamma_n(\delta) \cup \Gamma_{n+1}(\delta)\) to some \(\pi_n\gamma \in \Gamma_{n+1} \cup \Gamma_n\) that satisfies
			\begin{align*}
				\diam(\pi_n\gamma) \ge \diam(\gamma) - 2\dpaths(\gamma, \hat\gamma) \ge \delta - 2\cdot 2^{-n} = \delta - (d_n-d_{n+1}) \ge d_{n+1}
			\end{align*}
			i.e. \(\pi_n\gamma \in \Gamma_{n+1}(\delta - d_n+d_{n+1}) \cup \Gamma_n(\delta-d_n+d_{n+1})\). By induction, for \(m > n\) and \(\delta > d_n\) the composed matching \(\pi_{m, n} := \pi_{m-1}\pi_{m-2} \ldots \pi_{n} \subset \Gamma_m \times \Gamma_n\) matches \(\Gamma_n(\delta)\) perfectly with a subset of \(\Gamma_m(\delta - d_n + d_m) \subset \Gamma_m(\delta - d_n)\) and, conversely, matches \(\Gamma_m(\delta)\) perfectly with a subset of \(\Gamma_n(\delta - d_n+d_m) \subset \Gamma_n(\delta - d_n)\). We thus get
			\begin{align}\label{eqn:Gamman-sizes}
				|\Gamma_n(\delta + d_n)| \le |\Gamma_m(\delta)| \le |\Gamma_n(\delta - d_n)| < \infty, \qquad \forall m \ge n, \delta > d_n,
			\end{align}
			thus \(N(\delta) := \limsup_{m \to \infty} |\Gamma_m(\delta)|\) is bounded. Note that the map \(\delta \mapsto N(\delta) \in \bN\) is monotonic in \(\delta > 0\), thus the set \(D \subset \bR_+\) of its discontinuity points is discrete. For \(\delta \in \bR_+ \setminus D\) we thus may fix \(\varepsilon \in (0, \delta)\) such that \(N(\delta + u) = N(\delta)\) for every \(u \in [-\varepsilon, \varepsilon]\). Let \(n_\delta \in \bN\) be such that \(d_{n_\delta} \le \frac{\varepsilon}{2}\), and \(|\Gamma_{n_\delta}(\delta)| = N(\delta)\). By first applying the upper bound and then the lower bound in~\eqref{eqn:Gamman-sizes} (which is allowed since \(\delta-d_{n_\delta} \ge d_{n_\delta}\)) we get
			\begin{align*}
				|\Gamma_m(\delta+d_{n_\delta})| \le |\Gamma_n(\delta)| \le |\Gamma_n(\delta - d_{n_\delta})| \le |\Gamma_m(\delta - 2d_{n_\delta})|, \qquad \forall m \ge n \ge n_\delta.
			\end{align*}
			Since \(2d_{n_\delta} \le \varepsilon\). there exist \(m_+, m_- \ge n\) for which \(|\Gamma_{m_+}(\delta + d_{n_\delta})| = N(\delta) = |\Gamma_{m_-}(\delta-2d_{n_\delta})|\). This shows \(|\Gamma_n(\delta)| = |\Gamma_n(\delta - d_{n_\delta})| = N(\delta)\) for every \(n \ge n_\delta\), therefore \(\pi_{n, n_\delta}\) induces a perfect matching between \(\Gamma_{n_\delta}(\delta)\) and \(\Gamma_n(\delta)\).  A simple induction over \(m > n > n_\delta\) shows that the map \(\Gamma_{n_\delta}(\delta) \to \Gamma_n(\delta)\) given by \(\gamma \mapsto \pi_{n, n_\delta}\gamma\) is a bijection that satisfies
			\begin{align*}
				\dpaths(\pi_{n, n_\delta}\gamma, \pi_{m, n_\delta}\gamma) = \dpaths(\pi_{n, n_\delta}\gamma, \pi_{m, n}(\pi_{n, n_\delta}\gamma) \le \sum_{k = n}^{m-1} 2^{-k} < 2^{-n+1}, \qquad \forall \gamma \in \Gamma_{n_\delta}(\delta).
			\end{align*}
			Thus, for each \(\gamma \in \Gamma_{n_\delta}(\delta)\) the sequence \((\pi_{n, n_\delta}\gamma)_{n > n_\delta}\) is Cauchy in the complete space \(P\), therefore it admit a limit \(\hat \gamma := \lim_{n \to \infty} \pi_{n, n_\delta}\gamma \in \pathset(P)\). Define \(\hat\Gamma \in \Pathset(P)\) by\footnote{In general \(\Gamma\) is a multiset, thus for \(u, \delta > 0\) such that \(n_\delta \le n_u\), and \(\eta \in \Gamma_{n_u}(u), \gamma \in \Gamma_{n_{\delta}}(\delta)\), we identify \(\hat\eta = \hat\gamma\) if \(\eta = \pi_{n_u, n_\delta}\gamma\).}
			\begin{align*}
				\hat\Gamma := \bigcup_{\substack{\delta > 0 \\ \delta \notin D}}\Gamma(\delta), \qquad \text{where} \qquad \Gamma(\delta) := \{\hat \gamma \cond \gamma \in \Gamma_{n_\delta}(\delta)\}.
			\end{align*}
			For \(n \ge N_\delta\) we then have \(\dPaths(\hat\Gamma(\delta), \Gamma_n(\delta)) = \max_{\gamma \in \Gamma_n}\dpaths(\hat\gamma, \pi_{n, n_\delta}\gamma) \xrightarrow{n \to \infty} 0\). Since furthermore every \(\Gamma \in \Pathset(P)\) satisfies \(\dPaths(\Gamma, \Gamma(\delta)) \le \delta\), the triangle inequality gives
			\begin{align*}
				\dPaths(\hat\Gamma, \Gamma_n) \le \delta + \dPaths(\hat\Gamma(\delta), \Gamma_n(\delta)) + \delta \xrightarrow{n \to \infty} 2\delta, \qquad \forall \delta \in \bR_+\setminus D.
			\end{align*}
			Since \(D\) is a discrete subset of \(\bR_+\), \(\delta\) can be taken arbitrarily small, thus we conclude \(\lim_{n \to \infty} \Gamma_n = \hat\Gamma\), i.e. the Cauchy sequence \((\Gamma_n)_{n \in \bN}\) converges in \(\Pathset(P)\).
		\end{proof}
	\subsection{Annulus crossing properties}
	Let us begin with listing some basic properties of annulus crossings made by a curve without proofs.
	\begin{lem}\label{thm:crossing-properties}
		Fix a parametrized curve \(\gamma : I \to \mathcal X\), a point \(x \in \cX\), time instances \(s, t \in I\), and radii \(0 < r < R < \infty\).
		\begin{enumerate}[label=(\roman*)]
			\item\label{item:crossing-large-diameter} If \(\gamma \cap B_x(r) \ne \emptyset\) and \(\diam(\gamma) \ge 2R\), then \(\crossings{r}{R}{x}{\gamma} \ne \emptyset\).
			\item\label{item:crossing-small-diameter-empty} If \(\diam(\gamma) < R-r\), then \(\crossings{r}{R}{x}{\gamma} = \emptyset\).
			\item\label{item:crossings-disjoint} The set \(\crossings{r}{R}{x}{\gamma}\) consists of pairwise disjoint intervals: for any distinct elements \((a, b), (a', b') \in \crossings{r}{R}{x}{\gamma}\), we have \((a, b) \cap (a', b') = \emptyset\).
			\item\label{item:crossing-containment} If \(\gamma(s) \in \overline{B_x(r)}\) and \(\gamma(t) \notin B_x(R)\), then there exists \(a, b \in [\min(s, t), \max(s, t)]\) such that \((a, b) \in \crossings{r}{R}{x}{\gamma}\)
		\end{enumerate}
	\end{lem}
	The~\cref{item:crossing-large-diameter} of~\Cref{thm:crossing-properties} is the only result from this subsection that is used in the proof of the forward implication ``regular \(\implies\) precompact'' in~\Cref{thm:precompact-iff-regular}. Hence, a reader interested only in this implication may skip directly to~\Cref{sec:proof-of-regular-iff-precompact}.
	
	The converse implication ``precompact \(\implies\) regular'' in~\Cref{thm:precompact-iff-regular} at least requires any collection of paths \(\Gamma \in \Pathset\) to cross each annulus only finitely many times:
	\begin{lem}\label{thm:finite-crossings}
		Every \(\Gamma \in \Pathset, x \in \cX\) and \(0 < r < R < \infty\) satisfy \(\ncrossings rRx\Gamma < \infty.\)
	\end{lem}
	\begin{proof}
		Towards a contradiction, assume \(\ncrossings{r}{R}{x}{\Gamma} = \infty\). Since \(\Gamma(R-r)\) is a finite set by definition of \(\Pathset\), and by \Cref{thm:crossing-properties}\ref{item:crossing-small-diameter-empty} any curve in \(\Gamma \setminus \Gamma(R-r)\) does not cross \(\annulus{r}{R}{x}\), we can find \(\gamma \in \Gamma(R-r)\) such that \(\ncrossings{r}{R}{x}{\gamma} = \infty\). Fix a parametrization \(\gamma : I \to \cX\). By~\Cref{thm:crossing-properties}\ref{item:crossings-disjoint} we can find a sequence \((a_j, b_j)_{j \in \bN}\) of pairwise disjoint intervals in \(\crossings{r}{R}{x}{\gamma}\). Since \(I\) is compact, passing to a subsequence we may assume that the limit \(a := \lim_{j \to \infty} a_j\) exists. Since \(I\) has finite length, \(\sum_{j \in \bN}b_j-a_j < \infty\), so \(\lim_{j \to \infty} b_j = a\). However, \(|d(x, \gamma(b_j)) - d(x, \gamma(a_j))| = R-r > 0\), which contradicts continuity of \(\gamma\) at \(a\). By contradiction, we conclude that \(\ncrossings rRx\Gamma < \infty\).
	\end{proof}
	Given a hypothetical limit \(\mu\) of a non-regular sequence \((\mu_n)_{n \in \bN}\) in \(\cM_1(\cX)\), we will use Portmanteau's theorem (\Cref{thm:Portmanteau}) to find an estimate \(p > 0\) independent of \(N \in \bN\) satisfying
	\begin{align*}
		\mu[\ncrossings{r}{R}{x}{\Gamma} \ge N] \ge \limsup_{n \to \infty}\mu_n[\ncrossings{r}{R}{x}{\Gamma} \ge N] \ge p > 0,
	\end{align*} 
	which contradicts the above~\Cref{thm:finite-crossings}. To get the estimates given by Portmanteau's theorem in the correct direction, the sets \(\{\Gamma \in \Pathset \cond \ncrossings{r}{R}{x}{\Gamma} \ge N\}\) have to be closed subsets of \(\Pathset\) (\Cref{thm:closed-crossing-set}). When proving this, the following characterization of curves with a bounded number of annulus crossings is helpful.
	\begin{lem}\label{thm:separating-sequence}
		Let \(\gamma : [0, 1] \to \cX\) be a compact curve, and let \(x \in \cX\) and \(0 < r < R < \infty\). We have \(\ncrossings{r}{ R}{x}{\gamma} \le n\) if and only if there exist time instances \(0 = s_0 < s_1 < \ldots < s_n < s_{n+1} = 1\) such that \(\gamma[s_j, s_{j+1}]\) does not cross \(\annulus rRx\) for any \(j \in \{0, 1, 2, \ldots, n\}\).
	\end{lem}
	\begin{proof}
		First, suppose \(0 = s_0 < s_1 < \ldots < s_{n+1} = 1\) is as in the statement. For each interval \((a, b) \in \crossings rRx\gamma\), let \(j_a \in \{1, 2, \ldots, n\}\) be the smallest index such that \(a \in [s_{j_a-1}, s_{j_a}]\). By~\Cref{thm:crossing-properties}\ref{item:crossings-disjoint} any distinct \((a, b), (a', b') \in \crossings{r}{R}{x}{\gamma}\) are disjoint, hence we may without loss of generality assume \(b < a'\). In particular, \([s_{j_{a}-1}, s_{j_{a'}}]\) contains the intervals \((a, a') \supset (a, b) \in \crossings{r}{R}{x}{\gamma}\), so by~\Cref{thm:crossing-properties}\ref{item:crossing-containment} \(\gamma[s_{j_a-1}, s_{j_{a'}}]\) crosses \(\annulus{r}{R}{x}\). This is possible only if \(j_{a'} \ne j_a\), so \((a, b) \mapsto j_a\) injectively maps \(\crossings{r}{R}{x}{\gamma}\) to \(\{1, 2, \ldots, n\}\), proving \(\ncrossings{r}{R}{x}{\gamma} \le n\).
		
		For the converse, assume towards contradiction that for some \(n < \ncrossings{r}{R}{x}{\gamma}\) there exist time instances \(0 = s_0 < s_1 < \ldots <  s_{n+1} = 1\) as in the statement. Since \(\gamma[s_n, 1]\) does not cross \(\annulus{r}{R}{x}\), by~\Cref{thm:crossing-properties}\ref{item:crossing-containment} every \(a \in [s_n, 1]\) satisfies \((a, b) \notin \crossings{r}{R}{x}{\gamma}\) for every \(b > a\). Thus by pigeonhole principle there exists two distinct \((a, b), (a', b') \in \crossings{r}{R}{x}{\gamma}\) such that \(s_{j-1} \le a \le a' \le s_j\) for some \(j \in \{1, 2, \ldots, n\}\). Since by~\Cref{thm:crossing-properties}\ref{item:crossings-disjoint} \((a, b)\) and \((a', b')\) are disjoint, we have \([a, b] \subset [a, a'] \subset [s_{j-1}, s_j]\), so by~\Cref{thm:crossing-properties}\ref{item:crossing-containment} \(\gamma[s_{j-1}, s_j]\) crosses \(\annulus{r}{R}{x}\), which is a contradiction.
	\end{proof}
	\begin{lem}\label{thm:closed-crossing-set}
		Let \(\PATHSET \in \{\pathset, \Pathset\}\). For each \(x \in \cX, 0 < r < R < \infty\), and \(N \in \bN\), the set \(\{\bgamma \in \PATHSET \cond \ncrossings{r}{R}{x}{\bgamma} \ge N\}\) is a closed subset of \(\PATHSET\).
	\end{lem}
	\begin{proof}
		Take \(\gamma \in \pathset\) satisfying \(\ncrossings{r}{R}{x}{\gamma} < N\) for some \(N \in \bN\), and let \(0 = s_0 < s_1 < \dots < s_N = 1\) be a sequence from \Cref{thm:separating-sequence} for \(\gamma\). Since \(\gamma[s_j, s_{j+1}]\) does not cross the annulus \(\annulus rRx\), we can find \(\rho_j \in \{r, R\}\) such that \(\gamma[s_j, s_{j+1}] \cap \partial B_x(\rho_j) = \emptyset\). Since \(\gamma[s_j, s_{j+1}]\) is compact, we can find \(\delta_j > 0\) such that \(\dist(\gamma(t), \partial B_x(\rho_j)) \ge \delta_j\) for every \(t \in [s_j, s_{j+1}]\). If \(\delta := \min_j \delta_j > 0\), then for any \(\gamma' \in \pathset\) satisfying \(\dpaths(\gamma, \gamma') < \delta\) reverse triangle inequality implies \(\gamma'[s_j, s_{j+1}] \cap \partial B_x(\rho_j) = \emptyset\). We conclude that \(\gamma'[s_j, s_{j+1}]\) does not cross \(\annulus rRx\) for any \(j \in \{0, 1, \ldots, N-1\}\), hence by \Cref{thm:separating-sequence} we have \(\ncrossings{r}{R}{x}{\gamma'} < N\). This proves the case \(\PATHSET = \pathset\).
		
		Next, take \(\Gamma \in \Pathset\) satisfying \(\ncrossings{r}{R}{x}{\Gamma} < N\) for some \(N \in \bN\), and let \(\delta_0 := \frac{R-r}{3}\). By the above, for each \(\gamma \in \Gamma(\delta_0)\) we can find \(\delta_\gamma \in (0, \delta_0)\) such that \(\ncrossings{r}{R}{x}{\gamma'} \le \ncrossings{r}{R}{x}{\gamma}\) whenever \(\dpaths(\gamma, \gamma') < \delta_\gamma\). Let \(\delta := \min_{\gamma \in \Gamma(\delta_0)} \delta_\gamma\), and take \(\Gamma' \in \Pathset\) satisfying \(d(\Gamma, \Gamma') < \delta\). Since \(\delta < R-r\), every \(\gamma' \in \Gamma'(R-r)\) can be matched with some \(\gamma \in \Gamma\) so that \(\dpaths(\gamma, \gamma') < \delta\). In particular, since \(\delta < \delta_0\), we get
		\begin{align*}
			\diam(\gamma) \ge \diam(\gamma') - 2\dpaths(\gamma, \gamma') \ge (R-r) - 2\delta_0 = \delta_0,
		\end{align*}
		so \(\gamma \in \Gamma(\delta_0)\). Therefore, since \(\dpaths(\gamma, \gamma') < \delta_\gamma\), we have \(\ncrossings{r}{R}{x}{\gamma'} \le \ncrossings{r}{R}{x}{\gamma}\). On the other hand, by \Cref{thm:crossing-properties}\ref{item:crossing-small-diameter-empty} each \(\gamma' \in \Gamma'\setminus\Gamma'(R-r)\) satisfies \(\ncrossings{r}{R}{x}{\gamma'} = 0\). We thus get
		\begin{align*}
			\ncrossings{r}{R}{x}{\Gamma'} = \sum_{\gamma' \in \Gamma'(R-r)} \ncrossings{r}{R}{x}{\gamma'} \le \sum_{\gamma \in \Gamma(\delta_0)} \ncrossings{r}{R}{x}{\gamma} \le \ncrossings{r}{R}{x}{\Gamma} < N.
		\end{align*}
		Since the above holds for any \(\Gamma'\) satisfying \(d(\Gamma, \Gamma') < \delta\), the set \(\{\Gamma \in \Pathset \cond \ncrossings{r}{R}{x}{\Gamma} < N\}\) is open, so its complement \(\{\Gamma \in \Pathset \cond \ncrossings{r}{R}{x}{\Gamma} \ge N\}\) is closed.
	\end{proof}
	
	\subsection{Proof of~\Cref{thm:precompact-iff-regular}} \label{sec:proof-of-regular-iff-precompact}
	Recall from the proof sketch of~\Cref{thm:precompact-iff-regular} presented in the introduction that we want to approximate paths by piecewise geodesics. We define a concatenation of a finite sequence \((\gamma_j)_{j=1}^k\) of curves \(\gamma_j : [j-1, j] \to \cX\) satisfying \(\gamma_j(j) = \gamma_{j+1}(j)\) for every \(j \in \{1, \ldots, k-1\}\) by
	\begin{align*}
		\bigg(\prod_{j=1}^k \gamma_j\bigg)(t) := \gamma_{\lceil t \rceil}(t), \qquad t \in [0, k].
	\end{align*}
	Reparametrization of each \(\gamma_j\) by an increasing homeomorphism \(\sigma_j : [j-1, j] \to [j-1, j]\) corresponds to the reparametrization \(\sigma : [0, k] \to [0, k]\) of \(\prod_{j=1}^k \gamma_j\) defined as \(\sigma(t) := \sigma_{\lceil t \rceil}(t)\), thus \(\prod_{j=1}^k \gamma_j\) is well defined as an element in the space of unparametrized curves \(\pathset\).
	
	To organize the proof of~\Cref{thm:precompact-iff-regular}, we first prove the implication \ref{item:Mweakreg} \(\implies\) \ref{item:Mprecompact} for the cases \(\PATHSET = \pathset\) (\Cref{thm:regular-paths}) and \(\PATHSET = \Pathset\) (\Cref{thm:regular-Paths}) when the sets  \(\compactX_\varepsilon\) can be chosen to be compact. Then we combine and extend these results to finish the proof of~\Cref{thm:precompact-iff-regular}.
	\begin{prop}\label{thm:regular-paths}
		Let \((\cX, d)\) be a locally path connected metric space. Suppose \(M = \{\mu_n\}_{n\in\bN} \subset \pathset\) is weakly regular, and for every \(\varepsilon > 0\) there exists a compact set \(\compactX_\varepsilon \subset \cX\) such that
		\begin{align}\label{eqn:paths-asymptotically-tight}
			\liminf_{n \to \infty}\mu_n\big(\pathset(B_{\compactX_\varepsilon}(\delta))\big) \ge 1-\varepsilon, \qquad \forall \varepsilon, \delta > 0.
		\end{align}
		Then \((\mu_n)_{n \in \bN}\) is a precompact sequence.
	\end{prop}
	Note that this already generalizes the precompactness result in~\cite[Theorem 1.2]{aizenman1999holder}.
	\begin{proof}
		For \(x \in \cX\) and \(r > 0\), denote by \(U_x(r)\) the path connected component of \(B_x(r)\) containing \(x\). Since \(\cX\) is locally path connected, \(U_x(r)\) is an open neighbourhood of \(x\). Fix \(\varepsilon > 0\), and given \(\delta > 0\) let \(R_\delta = \frac{6}{\delta}\). For each \(x \in \compactX_\varepsilon\), let \(r_x = r_x(\delta) \in (0, R_\delta)\) be such that~\eqref{eqn:path-precompactness} holds:
		\begin{align}\label{eqn:regularity-proof}
			\lim_{N \to \infty}\sup_{\mu \in M}\mu[\ncrossings{r_x}{R_\delta}{x}{\gamma} \ge N] = 0.
		\end{align}
		By compactness, there exists a finite set \(F_\delta \subset \compactX_\varepsilon\) such that \(\{U_x(r_x)\}_{x \in F_\delta}\) is an open covering of \(\compactX_\varepsilon\). We may thus fix \(r = r(\delta)> 0\) such that \(B_{\compactX_\varepsilon}(r) \subset \bigcup_{x \in F_\delta}U_x(r_x)\). By~\eqref{eqn:regularity-proof}, we may fix \(N_\delta \in \bN\) such that
		\begin{align*}
			\max_{x \in F_\delta}\mu[\ncrossings{r_x}{R_\delta}{x}{\gamma} > N_\delta] \;<\; \frac{\varepsilon}{2|F_\delta|}, \qquad \forall \mu \in M,
		\end{align*}
		which by union bound implies
		\begin{align}\label{eqn:sample-crossing-bound-path}
			\mu\Big[\bigcap_{x \in F_\delta}\big\{\ncrossings{r_x}{R_\delta}{x}{\gamma} \le N_\delta\big\}\Big] \;\ge\; 1 - \sum_{x \in F_\delta}\mu\big[\ncrossings{r_x}{R_\delta}{x}{\gamma} > N_\delta\big] \;\ge\; 1-\frac{\varepsilon}{2}, \quad \forall \mu \in M.
		\end{align}
		For each subset \(Z \subset F_\delta\) and points \(z, w \in Z\) which are in the same path connected component of \(\bigcup_{x \in Z} U_x(r_x)\), fix a curve \(\eta^Z_{z,w} \in \pathset(\bigcup_{x \in Z} U_x(r_x))\) from \(z\) to \(w\). For \(n \in \bN\), consider the following finite subset of \(\pathset(\cX)\):
		\begin{align}\label{eqn:PRn-def}
			P_\delta(n) := \bigcup_{\ell = 0}^n \bigg\{\prod_{j=1}^\ell \eta^{Z_j}_{z_j, z_{j+1}} \cond \{z_j, z_{j+1}\} \subset Z_j \subset F \;\forall 1 \le j \le \ell\bigg\} \subset \pathset(\bigcup_{x \in F} U_x(r_x))
		\end{align}
		Let \(\tilde \gamma : [0, 1] \to \cX\) be a representative of a curve \(\gamma \in \pathset(B_{\compactX_\varepsilon}(r))\) satisfying \(\ncrossings{r_x}{R_\delta}{x}{\gamma} \le N_\delta\) for every \(x \in F_\delta\). Let \(t_1 = 0\), and choose points \(x_j \in F_\delta\) and times \(t_j \in [0, 1]\) inductively as follows:
		\begin{align*}
			x_j \in \{x \in F_\delta \cond \tilde\gamma(t_j) \in U_x(r_x)\}, \qquad t_{j+1} = \inf\Big(\big\{t \ge t_j \cond \gamma(t) \in \partial B_{R_\delta}(x_j)\big\} \cup \{1\}\Big);
		\end{align*}
		the points \(x_j\) exist since \(\{U_x(r_x)\}_{x \in F_\delta}\) covers \(B_{\compactX_\varepsilon}(r)\). Let \(\ell \in \bN\) be the first index for which \(t_{\ell+1} = 1\). For \(j \in \{1, \ldots, \ell\}\) denote by \(\gamma_j \in \pathset(B_{\compactX_\varepsilon}(r))\) the element represented by \(\tilde\gamma|[t_j, t_{j+1}]\), so that we have \(\gamma = \prod_{j=1}^{\ell} \gamma_j\). By~\Cref{thm:crossing-properties} \ref{item:crossing-containment}, for each \(j \in \{1, \ldots, \ell-1\}\) the curve \(\gamma_j\) crosses \(\annulus{r_{x_j}}{R_\delta}{x_j}\), therefore \(\ell-1 \le \sum_{x \in F_\delta}\ncrossings{r_x}{R_\delta}{x}{\gamma} \le N_\delta |F_\delta|\). 
		
		Denote by \(Z(\gamma_j)\) the set of elements \(x \in F\) for which \(\gamma_j \cap U_x(r_x) \ne \emptyset\); in particular, by construction we have \(\{x_j, x_{j+1}\} \in Z(\gamma_j)\). Note that the image of the path \(\gamma_j\) is covered by the path connected sets \((U_z(r_z))_{z \in Z(\gamma_j)}\). Therefore, \(\bigcup_{z \in Z(\gamma_j)} U_{z_j}(r_{z})\) is path connected, so the path \(\eta^{Z(\gamma_j)}_{x_j, x_{j+1}}\) exists. For any \(y, y' \in \bigcup_{x \in Z(\gamma_j)} U_{x}(r_{x})\) there exist \(z, z' \in Z(\gamma_j)\) such that \(y \in B_y(r_z)\) and \(y' \in B_{y'}(r_{z'})\), thus we may bound
		\begin{align*}
			d(y, y') \le \dist(y, \gamma_j) + \diam(\gamma_j) + \dist(y', \gamma_j) \le 2r_z + 2R_\delta + 2r_{z'} < 6R_\delta = \delta.
		\end{align*}
		Defining \(\eta^\gamma := \prod_{j=1}^\ell \eta^{Z(\gamma_j)}_{x_j, x_{j+1}} \in P_\delta(\ell) \subset P_\delta(N_\delta|F_\delta|+1)\), we thus conclude
		\begin{align}\label{eqn:geodesic-approximation-distance}
			\dpaths(\gamma, \eta^\gamma) \le \max_{1 \le j \le \ell} \dpaths(\gamma^j, \eta^{Z(\gamma_j)}_{x_j, x_{j+1}}) < \max_{1 \le j \le \ell} \diam\Big(\bigcup_{z \in Z(\gamma_j)}U_z(r_z)\Big) \le \delta,
		\end{align}
		This shows that \(\dist(\gamma, P_\delta(N_\delta|F_\delta|+1)) < \delta\), thus by monotonicity of measures we get
		\begin{align*}
			\liminf_{n \to \infty}\mu_n\Big(B_{P_\delta(N_\delta|F_\delta|+1)}(\delta)\Big) \ge\;& \liminf_{n \to \infty}\mu_n\bigg[\big\{\gamma \in \pathset(B_{\compactX_\varepsilon}(r))\} \cap \bigcap_{x \in F_u}\big\{\ncrossings{r_x}{R_\delta}{x}{\gamma} \le N_\delta\}\bigg] \\
			\ge\;& 1-\varepsilon.
		\end{align*}
		Choose \(K_{\varepsilon}^{1/n} := P_{1/n}(N_{1/n}|F_{1/n}|+1)\). Since \(r_x < R_{1/n} = \frac{1}{6n}\) for every \(x \in F_{1/n}\), by the last inclusion in~\eqref{eqn:PRn-def} we have \(K^{1/n}_\varepsilon \subset \pathset(B_{\compactX_\varepsilon}(\frac{1}{6n}))\). Since each set \(K^{1/n}_\varepsilon\) is finite, by~\Cref{thm:paths-near-complete-are-complete} the set \(K_\varepsilon := \pathset(\compactX_\varepsilon) \cup \bigcup_{n=1}^\infty K_\varepsilon^{1/n}\) is complete. Precompactness of \(M\) follows from~\Cref{thm:approximate-tightness-implies-precompactness}.
	\end{proof}
	
		\begin{remark}\label{rem:tortuosity-bounds}
			Recall that tortuosity of a curve \(\gamma \in \pathset\) is the function \(\ell \mapsto \tortuosity(\gamma, \ell)\) whose value is the minimal size of a partition of \(\gamma\) into segments having diameter at most \(\ell\). Generalizing \cite[Lemma 4.1]{aizenman1999holder}, for a compact metric space \(\cX\) a set \(K \subset \pathset(\cX)\) is precompact if and only if \(\sup_{\gamma \in K}\tortuosity(\gamma, \ell) < \infty\) for every \(\ell > 0\). 
			
			This can be used to give another proof of~\Cref{thm:regular-paths} in the case where \(\cX\) is locally compact. The key insight is the following: given \(\varepsilon > 0\), by local compactness there exists \(\delta > 0\) such that \(\overline{B_{Y_\varepsilon}(\delta)}\) is a compact set, thus we may without loss of generality assume
			\begin{align}\label{eqn:path-image-tightness}
				\liminf_{n \to \infty} \mu_n(\pathset(Y_\varepsilon)) \ge 1-\varepsilon.
			\end{align} 
			Take a decreasing sequence \(\delta_k \searrow 0\). Let \(R_k := R_{\delta_k}, F_k := F_{\delta_k} \subset \compactX_\varepsilon\), and \(r_k := r(\delta_k)\) be as in the beginning of the proof of~\Cref{thm:regular-paths}. By regularity, choosing \(N_k \in \bN\) sufficiently large we get
			\begin{align}\label{eqn:tortuosity-probability-good}
				\mu\Big[\bigcap_{x \in F_k}\{\ncrossings{r_x}{R_k}{x}{\gamma} \le N_k\}\Big] \ge 1-2^{-k}\varepsilon, \qquad \forall \mu \in M.
			\end{align}
			Note that the event \(\tilde K_\varepsilon := \bigcap_{k \in \bN}\bigcap_{x \in F_k}\{\ncrossings{r_x}{R_k}{x}{\gamma} \le N_k\}\) satisfies \(\mu(\tilde K_\varepsilon) \ge 1-\varepsilon\) for every \(\mu \in M\), and implies tortuosity bounds \(\tortuosity(\gamma, 2R_k) \le N|F_k|\) for every \(\gamma \in \pathset(B_{\compactX_\varepsilon}(r_k))\). In particular, for each \(\ell > 0\) the tortuosity \(\tortuosity(\gamma, \ell)\) is uniformly bounded bounded for \(\gamma \in \tilde K_\varepsilon \cap \pathset(Y_\varepsilon) =: K_{2\varepsilon}\). Thus, \(\overline{K_{2\varepsilon}}\) is compact by~\cite[Lemma 4.1]{aizenman1999holder}, and it satisfies \(\liminf_{n \to \infty}\mu_n(\overline{K_{2\varepsilon}}) \ge 1-2\varepsilon\).
			Prokhorov's~\Cref{thm:prokhorov} then gives precompactness of \(M\).
			
			One may hope that with the assumptions of~\Cref{thm:regular-paths} a similar strategy would give asymptotic tightness and thus precompactness of \(M\), however this is not the case.
			The issue is that \(B_{K_{2\varepsilon}}(\delta)\) may not contain \(\tilde K_\varepsilon \cap P(B_{\compactX_\varepsilon}(u))\) for any \(0 < u < \delta\), so~\eqref{eqn:paths-asymptotically-tight} can not be used to bound \(\mu(B_{K_{2\varepsilon}}(\delta))\) for any \(\mu \in M\). It is not immediately clear how to modify the argument to get asymptotic tightness, thus~\Cref{thm:regular-paths} (and by extension~\Cref{thm:regular-Paths} and~\Cref{thm:precompact-iff-regular}) serve as examples where approximate tightness is easier to verify.
		\end{remark}
	
	\begin{prop}\label{thm:regular-Paths}
		Let \((\cX, d)\) be a locally path connected metric space. Suppose \(M = \{\mu_n\}_{n \in \bN} \subset \Pathset\) is \(\regconst\)-regular, and for every \(\varepsilon > 0\) there exists a compact set \(\compactX_\varepsilon \subset \cX\) such that such that 
		\begin{align*}
			\liminf_{n \to \infty}\mu_n\big(\pathset(B_{\compactX_\varepsilon}(\delta))\big) \ge 1-\varepsilon, \qquad \forall \delta > 0.
		\end{align*}
		Then \((\mu_n)_{n \in \bN}\) is a precompact sequence.
	\end{prop}
	\begin{proof}
		Fix \(\varepsilon > 0\). Given \(\delta > 0\), let \(R_\delta = \frac{\delta}{6}, (r_x > 0)_{x \in \cX}, F_\delta \subset \compactX_\varepsilon\), and \(r > 0\) be as in the proof of~\Cref{thm:regular-paths}. We can similarly find \(N_\delta \in \bN\) such that the analogue of \eqref{eqn:sample-crossing-bound-path} holds:
		\begin{align}\label{eqn:sample-crossing-bound}
			\mu\Big[\bigcap_{x \in F}\big\{\ncrossings{r_x}{R_\delta}{x}{\Gamma} \le N_\delta\}\Big] \;\ge\; 1-\frac{\varepsilon}{2}, \quad \forall \mu \in M.
		\end{align}
		Our aim is to approximate a path collection \(\Gamma \in \Pathset(B_{\compactX_\varepsilon}(r))\) satisfying \(\sum_{x \in F_\delta} \ncrossings{r_x}{R_\delta}{x}{\Gamma} \le n\) with an element in the finite set 
		\begin{align*}
			\operatorname{P}_{\!\delta}(n) := \{\Gamma \in \Pathset \cond |\Gamma| \le n, \gamma \in P_\delta(n+1)\;\forall \gamma \in \Gamma\},
		\end{align*}
		where \(P_\delta(n+1)\) is defined by~\eqref{eqn:PRn-def}. Note that every \(\gamma \in \Gamma\) also satisfies \(\sum_{x \in F_\delta}\ncrossings{r_x}{R_\delta}{x}{\gamma} \le n\), therefore, as in the proof of~\Cref{thm:precompact-iff-regular}, we may construct \(\eta^\gamma \in P_\delta(n+1)\) that satisfies \(\dpaths(\gamma, \eta^\gamma) < \delta\) from~\eqref{eqn:geodesic-approximation-distance}. If furthermore \(\diam(\gamma) \ge 2R\), by~\Cref{thm:crossing-properties}\ref{item:crossing-large-diameter} \(\gamma\) crosses at least one of the annuli \(\{\annulus{r_x}{R}{x}\}_{x \in F_\delta}\). As a result, we have \(|\Gamma(2R)| \le \sum_{x \in F_\delta}\ncrossings{r_x}{R}{x}{\Gamma} \le n\). We may thus define
		\begin{align*}
			\hat\Gamma := \{\eta^\gamma \cond \gamma \in \Gamma(2R)\} \in \operatorname{P}_{\!\delta}(n),
		\end{align*}
		The map \(\gamma \mapsto \eta^\gamma\) induces a perfect matching between \(\Gamma(2R)\) and \(\hat\Gamma\), hence we get
		\begin{align*}
			\dPaths(\Gamma, \hat \Gamma) \le \max\Big(\max_{\gamma \in \Gamma(4u)} \underbrace{\dpaths(\gamma, \eta^\gamma)}_{< 6R \le \delta}, \sup_{\gamma \in \Gamma \setminus \Gamma(4u)} \underbrace{\diam(\gamma)}_{\le 4R < \delta}\Big) < \delta.
		\end{align*}
		This shows that \(\dist(\Gamma, \operatorname{P}_{\!\delta}(n)) < \delta\), thus by~\eqref{eqn:sample-crossing-bound} we may bound
		\begin{align*}
			\liminf_{n \to \infty} \mu_n\Big(B_{\operatorname{P}_{\!\delta}(N_\delta|F_\delta|)}(\delta)\Big) \ge \liminf_{n \to \infty}\mu_n\bigg[\big\{\Gamma \in \Pathset(B_{\compactX_\varepsilon}(r))\} \cap \bigcap_{x \in F_\delta}\big\{\ncrossings{r_x}{R}{x}{\Gamma} \le N_\delta\}\bigg] \ge 1-\varepsilon.
		\end{align*}
		Choose \(\mathcal{K}_{\varepsilon}^{1/n} := \operatorname{P}_{\!1/n}(N_{1/n}|F_{1/n}|)\), and \(\mathcal{K}_\varepsilon := \Pathset(K_\varepsilon)\), where \(K_\varepsilon := \pathset(\compactX_\varepsilon) \cup \bigcup_{n \in \bN} K^{1/n}_\varepsilon\) from~\Cref{thm:regular-paths} was shown to be complete subspace of \(\pathset(\cX)\). By~\Cref{thm:pathsets-of-complete-paths-is-complete}, \(\mathcal{K}_\varepsilon\) is thus complete. Note that each \(\mathcal{K}^{1/n}_\varepsilon\) is a finite (thus compact) subset of \(\Pathset(K^{1/n}_\varepsilon) \subset \mathcal{K}_\varepsilon\), thus we conclude approximate tightness of \((\mu_n)_{n \in \bN}\). Precompactness follows from~\Cref{thm:approximate-tightness-implies-precompactness}.
	\end{proof}
	One may be inclined to say that in  Propositions~\ref{thm:regular-paths} and~\ref{thm:regular-Paths} we assume the images of the paths to be asymptotically tight, while in~\Cref{thm:precompact-iff-regular} they are assumed to be approximately tight. This is indeed true in the Hausdorff metric (defined below), thus by the implication \ref{item:approximate-tightness} \(\implies\) \ref{item:asymptotic-tightness} in~\Cref{thm:main-thm} they are equivalent assumptions.
	
	Given a metric space \(\cX\), denote by \(\compacts(\cX)\) the set of non-empty compact subsets of \(\cX\), and equip it with the Hausdorff metric \(\dHausdorff\) defined as
	\begin{align*}
		\dHausdorff(K_1, K_2) := \inf\{\varepsilon > 0 \cond K_2 \subset B_{K_1}(\varepsilon) \text{ and } K_1 \subset B_{K_2}(\varepsilon)\}.
	\end{align*}
	Denote by \(\imag : \pathset(\cX) \to \compacts(\cX)\) the map that takes \(\gamma \in \pathset(\cX)\) to \(\tilde \gamma[0, 1] \in \compacts(\cX)\), where \(\tilde \gamma : [0, 1] \to \cX\) is a path representing \(\gamma\). The final ingredient we need to prove~\Cref{thm:main-thm} is the following converse to the fact that for compact \(\cX\), the space \(\compacts(\cX)\) is also compact.
	\begin{lem}\label{thm:compact-union}
		Let \((\cX, d)\) be a metric space, and suppose \(\cA \subset \compacts(\cX)\) is compact. Then \(Y := \bigcup_{A \in \cA} A \subset \cX\) is compact.
	\end{lem}
	\begin{proof}
		Take a sequence \((x_n)_{n \in \bN}\) in \(Y\). For each \(n \in \bN\) we can take \(A_n \in \cA\) such that \(x_n \in A_n\). Since \(\cA\) is compact, there exists a subsequence \((A_{n_k})_{k \in \bN}\) converging to some \(A \in \cA\). Since \(A \subset \cX\) is compact, for every \(k \in \bN\) there exists \(y_k \in A\) such that \(d(x_{n_k}, y_k) = \dist(x_n, A) \le \dHausdorff(A_{n_k}, A)\). Now, \((y_k)_{k \in \bN}\) is a sequence in the compact set \(A\), thus it admits a subsequnece \((y_{k_j})_{j \in \bN}\) converging to some \(y \in A\). We thus get
		\begin{align*}
			d(x_{n_{k_j}}, y) \le d(x_{n_{k_j}}, y_{k_j}) + d(y_{k_j}, y) \le \dHausdorff(A_{n_{k_j}}, A) + d(y_{k_j}, y) \xrightarrow{j \to \infty} 0,
		\end{align*}
		so \((x_n)_{n \in \bN}\) has a subsequence converging to \(A \subset Y\).
	\end{proof}
	
	\precompactiffregular*
	\begin{proof}[Proof]
		(i) \(\implies\) (ii): Trivial.
		
		(ii) \(\implies\) (iii): We aim to show that every sequence \((\mu_n)_{n \in \bN}\) in \(M\) is precompact. Without loss of generality, assume
		\begin{align}\label{eqn:path-images-approximately-tight}
			\liminf_{n \to \infty} \mu_n(\PATHSET(B_{Y_\varepsilon^\delta})(\delta)) \ge 1-\varepsilon, \qquad \forall \varepsilon,\delta > 0;
		\end{align}
		otherwise we could pass to a subsequence which stays in the precompact set \(M_\varepsilon^\delta\) for some \(\varepsilon, \delta > 0\). Recall that completeness and compactness of a metric space \(Y\) both transfer to \(\compacts(Y)\). Therefore, if we define
		\begin{align*}
			\tilde \mu_n := \mu_n \circ \imag^{-1} \in \cM_1(\compacts(\cX)),
		\end{align*} then~\eqref{eqn:path-images-approximately-tight} implies approximate tightness of the sequence \((\tilde \mu_n)_{n \in \bN}\). By~\Cref{thm:main-thm} this is equivalent with asymptotic tightness of \((\tilde \mu_n)_{n \in \bN}\). For each \(\varepsilon > 0\) there thus exists a compact set \(\cA_\varepsilon \subset \compacts(\pathset)\) such that
		\begin{align*}
			\liminf_{n \to \infty} \tilde \mu_n(B_{\cA_\varepsilon}(\delta)) \ge 1-\varepsilon, \qquad \forall \delta > 0.
		\end{align*}
		By~\Cref{thm:compact-union}, the set \(\tilde\compactX_\varepsilon := \bigcup_{A \in \cA_\varepsilon} A \) is compact. For every \(\delta > 0\) we have \(B_{\cA_\varepsilon}(\delta) \subset \compacts(B_{\tilde\compactX_\varepsilon}(\delta))\), thus we get
		\begin{align*}
			\liminf_{n \to \infty}\mu_n(\pathset(B_{\tilde \compactX_\varepsilon}(\delta))) = \liminf_{n \to \infty}\tilde\mu_n(\compacts(B_{\tilde\compactX_\varepsilon}(\delta))) \ge \liminf_{n \to \infty} \tilde \mu_n(B_{\cA_\varepsilon}(\delta)) \ge 1-\varepsilon, \quad \forall \delta > 0.
		\end{align*}
		Precompactness of \(M\) thus follows from Propositions~\ref{thm:regular-paths} and~\ref{thm:regular-Paths}.
		
		(iii) \(\implies\) (i): Assume \(M \subset \cM_1(\PATHSET)\) is not regular. Then there exists \(x \in \cX, R > r > 0, p > 0\), and a sequence \((\mu_n)_{n \in \bN}\) in \(M\), such that
		\begin{align}\label{eqn:infinite-crossing-subsequence}
			\mu_n[N^{\bgamma}_{r,R}(x) \ge n] \ge p \qquad \forall n \in \bN.
		\end{align}
		Towards a contradiction, assume there exists a subsequence \((n_k)_{k \in \mathbb N}\) such that \((\mu_{n_k})_{k \in \bN}\) converges weakly to some \(\mu \in \cM_1(\PATHSET)\). By \Cref{thm:closed-crossing-set}, \(\{\bgamma \in \PATHSET \cond \ncrossings{r}{R}{x}{\bgamma} \ge N\}\) is a closed set. Hence, Portmanteau theorem (\Cref{thm:Portmanteau}) yields
		\begin{align}\label{eqn:crossing-probability-limit}
			\limsup_{k \to \infty}\mu_{n_k}[\ncrossings{r}{R}{x}{\bgamma} \ge N] \le \mu[\ncrossings{r}{R}{x}{\bgamma} \ge N] \xrightarrow{N \to \infty} 0,
		\end{align}
		where the last limit holds by monotone convergence theorem and the deterministic fact \(\ncrossings{r}{R}{x}{\bgamma} < \infty\) from \Cref{thm:finite-crossings}. On the other hand, by \eqref{eqn:infinite-crossing-subsequence}, for any \(k \in \mathbb N\) such that \(n_k \ge N\) we have
		\begin{align*}
			\mu_{n_k}[\ncrossings rRx{\bgamma} \ge N] \ge \mu_{n_k}[\ncrossings rRx{\bgamma} \ge n_k] \ge p, 
		\end{align*}
		which contradicts the limit~\eqref{eqn:crossing-probability-limit}. By contradiction, we conclude that the subsequence \((\mu_{n_k})_{k \in \bN}\) does not contain any converging subsequences, so \(M\) is not a precompact subset of \(\cM_1(\PATHSET)\).
	\end{proof}
	\subsection{Refinement of regularity condition} \label{sec:refinement-of-regularity-condition}
	In this subsection we prove~\Cref{thm:regular-curves-euclidean}.
	First, introduce the following notation for \(\gamma \in \pathset, \Gamma \in \Pathset\), and a subset \(S \subset \cX\):
	\begin{gather*}
		\scrossings{r}{R}{x}{\gamma}{S} := \{(a, b) \in \crossings{r}{R}{x}{\gamma} : \gamma[a, b] \cap S \ne \emptyset\},\\
		\nscrossings{r}{R}{x}{\gamma}{S} := |\scrossings{r}{R}{x}{\gamma}{S}|, \qquad \nscrossings{r}{R}{x}{\Gamma}{S} := \sum_{\gamma \in \Gamma} \nscrossings{r}{R}{x}{\gamma}{S}.
	\end{gather*}
	As a blueprint, we generalize the heuristic argument presented after the statement of~\Cref{thm:precompact-iff-regular} in the introduction.
	In the proof of the upcoming~\Cref{thm:regularity-on-bdd+closed=compact} below, the set denoted by \(S_0\) plays the role of \(\partial B_x(\frac{R+r}{2})\) in the heuristic argument.
	\begin{prop}\label{thm:regularity-on-bdd+closed=compact}
		Let \(\PATHSET \in \{\pathset, \Pathset\}\), and \((\cX,d)\) be a proper metric space, that is, every closed ball is compact. Then \(M \subset \cM_1(\PATHSET)\) is regular if and only if it is weakly regular.
	\end{prop}
	\begin{proof}
		Suppose \(M \subset \cM_1(\PATHSET)\) is not regular; the other implication is trivial. Then there exists a point \(x \in \cX\) and radii \(R>r>0\) such that
		\begin{align*}
			\lim_{N \to \infty}\sup_{\mu \in M}\mu[\ncrossings{r}{R}{x}{\bgamma} \ge N] = p_0 > 0.
		\end{align*}
		Let \(S_0 \subset \annulus{r}{R}{x}\) be a closed subset separating \(\partial B_x(r)\) from \(\partial B_x(R)\). Since \(S_0\) is also bounded, it is compact. Since each crossing \(\gamma[a, b], (a, b) \in \crossings{r}{R}{x}{\gamma}\), connects \(\partial B_x(r)\) to \(\partial B_x(R)\), we have \(\scrossings{r}{R}{x}{\gamma}{S_0} = \crossings{r}{R}{x}{\gamma}\), and consequently
		\begin{align*}
			\lim_{N \to \infty}\sup_{\mu \in M}\mu[\nscrossings{r}{R}{x}{\bgamma}{S_0} \ge N] = p_0.
		\end{align*}
		Given \(S_k\) we will construct a compact set \(S_{k+1} \subset S_k\) as follows. Fix a number \(\varepsilon_{k+1} > 0\). Since \(S_k\) is compact, we can find a finite cover \(\mathcal F_k\) of \(S_k\) such that \(\diam(S) \le \varepsilon_{k+1}\) for every \(S \in \mathcal F_k\); by replacing each \(S \in \mathcal F_k\) with \(\overline{S \cap S_k}\) we may without loss of generality assume each set \(S \in \cF_k\) to be closed and satisfy \(S \subset S_k\), so in particular \(S\) is compact. In such case, an \(\annulus{r}{R}{x}\)-crossing hits the set \(S_k\) if and only if it hits at least one of the sets in \(\cF_k\). Hence, for any \(\gamma \in \pathset\) we get \(\scrossings{r}{R}{x}{\gamma}{S_k} = \bigcup_{S \in \mathcal F_k} \scrossings{r}{R}{x}{\gamma}{S}\), and consequently
		\begin{align*}
			\nscrossings{r}{R}{x}{\bgamma}{S_k} \le \sum_{S \in \mathcal F_k}\nscrossings{r}{R}{x}{\bgamma}{S}.
		\end{align*}
		At least one of \(S \in \mathcal F_k\) thus has to satisfy \(\nscrossings{r}{R}{x}{\bgamma}{S} \ge \frac{\nscrossings{r}{R}{x}{\bgamma}{S_k}}{|\mathcal F_k|}\). As this holds for any \(\gamma \in \pathset\), we get the inclusion of events
		\begin{align*}
			\Big\{\bgamma \in \PATHSET : \nscrossings{r}{R}{x}{\bgamma}{S_k} \ge N\Big\} \subset \bigcup_{S \in \mathcal F_k} \bigg\{\bgamma \in \PATHSET : \nscrossings{r}{R}{x}{\bgamma}{S} \ge \frac{N}{|\mathcal F_k|}\bigg\}.
		\end{align*}
		Union bound and the above inclusion thus yield
		\begin{align*}
			\sum_{S \in \mathcal F_k}\sup_{\mu \in M}\mu\Big[\nscrossings{r}{R}{x}{\bgamma}{S} \ge \frac{N}{|\mathcal F_k|}\Big] \;\ge\; \sup_{\mu \in M} \mu[\nscrossings{r}{R}{x}{\bgamma}{S_k} \ge N].
		\end{align*}
		Taking the limit \(N \to \infty\) of both sides above we conclude that there exists \(S_{k+1} \in \mathcal F_k\) satisfying
		\begin{align}\label{eqn:many-set-crossings}
			\begin{split}
				\lim_{N \to \infty}\sup_{\mu \in M}\mu\big[\nscrossings{r}{R}{x}{\bgamma}{S_{k+1}} \ge N\big] \ge\;& \frac{\lim_{N\to \infty}\sup_{\mu \in M}\mu[\nscrossings{r}{R}{x}{\bgamma}{S_k} \ge N]}{|\mathcal F_k|} \\
				=\;& \frac{p_k}{|\mathcal F_k|} =: p_{k+1}.
			\end{split}
		\end{align}
		We have constructed a decreasing sequence of compact sets \((S_k)_{k \in \bN}\) satisfying \(\diam(S_k) \le \varepsilon_k\). By choosing the diameters \(\varepsilon_k\) so that \(\lim_{k \to \infty} \varepsilon_k = 0\) there exists a unique element \(y \in \bigcap_{k \in \bN} S_k \subset S_0\). Note that since \(S_0 \cap \partial \annulus{r}{R}{x} = \emptyset\), we have \(\rho := \dist(y, \partial\annulus{r}{R}{x}) > 0\).
		
		Take \(\varepsilon \in (0, \rho)\), and let \(k \in \bN\) be large enough so that \(\varepsilon_k < \varepsilon\). For \(\gamma \in \pathset\), take \((a, b) \in \scrossings{r}{R}{x}{\gamma}{S_k}\). Since \(\diam(S_k) \le \varepsilon_k < \varepsilon\), and \(y \in S_k\), we have \(S_k \subset B_y(\varepsilon)\), so we can find \(t \in (a, b)\) such that \(\gamma(t) \in S_k \subset B_y(\varepsilon)\). On the other hand, \(\gamma(a) \in \partial \annulus{r}{R}{x}\), so in particular \(\gamma(a) \notin B_y(\rho)\). By~\Cref{thm:crossing-properties}\ref{item:crossing-containment} we can thus find \(a', b' \in (a, b)\) such that \((a', b') \in \crossings{\varepsilon}{\rho}{y}{\gamma}\). As the intervals in \(\crossings{r}{R}{x}{\gamma}\) are disjoint by~\Cref{thm:crossing-properties}\ref{item:crossings-disjoint}, the map \((a, b) \mapsto (a', b')\) sends elements of \(\scrossings{r}{R}{x}{\gamma}{S_k}\) injectively to \(\crossings{\varepsilon}{\rho}{y}{\gamma}\), so we have \(\ncrossings{\varepsilon}{\rho}{y}{\gamma} \ge \nscrossings{r}{R}{x}{\gamma}{S_k}\). As this holds for any \(\gamma \in \pathset\), we get
		\begin{align}\label{eqn:nonregular-annulus}
			\lim_{N \to \infty}\sup_{\mu \in M}\mu[\ncrossings{\varepsilon}{\rho}{y}{\bgamma} \ge N] \ge \lim_{N \to \infty}\limsup_{\mu \in M}\mu[\nscrossings{r}{R}{x}{\bgamma}{S_k} \ge N] \ge p_k > 0.
		\end{align}
		We conclude that the equation~\eqref{eqn:path-precompactness} with the choices \(x = y, r = \varepsilon\), and \(R = \rho\) fails for every \(\varepsilon \in (0, \rho)\), so \(M\) is not regular at \(y\). The result hence follows by contraposition. 
	\end{proof}
	Let us next highlight how to modify the above proof to get~\Cref{thm:regular-curves-euclidean}. Equation \eqref{eqn:many-set-crossings} provides a way to quantitatively estimate the values of \(p_k\) appearing in~\eqref{eqn:nonregular-annulus}; one needs to construct the covers \(\cF_k\) in a way that \(|\cF_k|\) depends only on \(k\). The proof of \Cref{thm:regular-curves-euclidean} below achieves this by using cubes which are inherently euclidean objects. A more natural and generalizable idea would be to use balls instead as follows. Choose \(S_0 = \partial B_x(\frac{R+r}{2})\), and \(\cF_k\) to consist of balls of radius \(\varepsilon_{k+1}\) with centers in \(S_k\). In \(\bR^d\), using for example Vitali's covering lemma one can find a constant \(c > 1\) depending only on \(d\) such that \(|\cF_k| \le c(\frac{\varepsilon_k}{\varepsilon_k+1})^{d-1}\), which from the recursive equation \(p_{k+1} = \frac{p_k}{|\cF_k|} = c^{-1}(\frac{\varepsilon_{k+1}}{\varepsilon_k})^{d-1}\) yields
	\begin{align*}
		p_k \ge c^{-k}\Big(\frac{\varepsilon_k}{\varepsilon_0}\Big)^{d-1}.
	\end{align*}
	If \(c > 1\), the term \(c^{-k}\) prevents us to prove \Cref{thm:regular-curves-euclidean}, however we can get arbitrarily close to it in the following sense. Fix a function \(g : \bR_+ \to \bR_+\) such that \(\lim_{r \to 0}g(r) = 0\). By choosing \(\varepsilon_k\) small enough so that \(g(\varepsilon_k) \le \frac{c^{-k}}{\varepsilon_0^{d-1}}\) we then get \(p_k \ge g(\varepsilon_k)\varepsilon_k^{d-1}\). With the choice \(\varepsilon = \varepsilon_k\),~\eqref{eqn:nonregular-annulus} becomes
	\begin{align*}
		\lim_{N \to \infty} \sup_{\mu \in M} \mu[\ncrossings{\varepsilon_k}{\rho}{y}{\Gamma} \ge N] \ge g(\varepsilon_k)\varepsilon_k^{d-1}.
	\end{align*}
	As the choice of \(g\) was arbitrary up to \(\lim_{r \to 0}g(r) = 0\), the contraposition shows the following.
	\begin{prop}\label{thm:ball-version}
		Let \(\PATHSET \in \{\pathset, \Pathset\}\). A subset \(M \subset \cM_1(\PATHSET(\bR^d))\) is regular if and only if for every \(x \in \bR^d\) and \(R > 0\) there exists a function \(g : \bR_+ \to \bR_+\) such that \(\lim_{r \to 0} g(r) = 0\), and
		\begin{align*}
			\lim_{N \to \infty}\sup_{\mu \in M}\mu[\ncrossings{r}{R}{x}{\bgamma} \ge N] = o(g(r)r^{d-1}), \qquad \text{as } r \to 0.
		\end{align*}
		\hfill \(\square\)
	\end{prop}
	The value \(d-1\) represents the box counting dimension of the set \(S_k \subset \partial B_x(\frac{R+r}{2})\). In fact, the proof of \Cref{thm:ball-version} above can be modified to work even if the constant \(c\) depends on \(k\), thus by bounding the \emph{lower} box counting dimension of the sets \(\partial B_x(\rho)\) one can prove analogues of \Cref{thm:ball-version} for more general metric spaces, for example Riemannian manifolds of dimension \(d\).
	
	Let us finish this section with the proof of \Cref{thm:regular-curves-euclidean}.
	\regulareuclideanthm*
	\begin{proof}[Proof of \Cref{thm:regular-curves-euclidean}]
		We will follow the structure and notations of the proof of \Cref{thm:regularity-on-bdd+closed=compact}. Suppose \(M \subset \cM_1(\PATHSET(\bR^d))\) is not regular.  For each \(z \in \annulus{r}{R}{x}\), let \(Q_z\) be an open cube containing the point \(z\) such that \(\overline {Q_z} \subset \annulus{r}{R}{x}\). These cubes cover the compact set \(\partial B_x(\frac{R+r}{2}) \subset \annulus{r}{R}{x}\), so we can find a finite subset \(F \subset \annulus{r}{R}{x}\) such that \((Q_z)_{z \in F}\) covers \(\partial B_x(\frac{R+r}{2})\). Since \(\bigcup_{z \in F} \overline{Q_z}\) separates \(\partial B_x(r)\) and \(\partial B_x(R)\), any \(\annulus{r}{R}{x}\) crossing \(\eta: (a, b) \to \bR^d\) hits at least one of the sets \(\overline{Q_z}\). Furthermore, as \(\eta(a) \notin \annulus{r}{R}{x} \supset \overline{Q_z}\), and for some \(t \in (a, b)\) we have \(\eta(t) \in \overline{Q_z}\), we conclude that \(\eta[a, b] \cap \partial Q_z \ne \emptyset\). This shows that also the set \(S_0 := \bigcup_{z \in F} \partial Q_z\) separates \(\partial B_x(r)\) and \(\partial B_x(R)\).
		
		Denote by \(\mathcal Q_z\) the set consisting of the \(2d\) faces of the cube \(Q_z\). In particular, \(\partial Q_z = \bigcup_{S \in \mathcal Q_z} S\), hence \(\cF_0 := \bigcup_{z \in F} \mathcal Q_z\) is a finite cover of \(S_0\). Choose a set \(S_1 \in \cF_0\) as in the proof of \Cref{thm:regularity-on-bdd+closed=compact} so that it satisfies~\eqref{eqn:many-set-crossings} for \(k = 0\), and denote by \(s_1\) the side length of \(S_1\). We may isometrically project \(S_1\) to a cube \(\pi S_1 \subset \bR^{d-1}\) so that \(\pi S_1 = \prod_{j=1}^{d-1} I^1_j\), where \(I^1_j \subset \bR\) are closed intervals of length \(s_1\).
		
		Suppose \(S_k \subset S_1\) is a closed cube with side length \(s_k\) which satisfies
		\begin{align}\label{eqn:many-cube-crossings}
			\lim_{N \to \infty}\sup_{\mu \in M} \mu[\nscrossings{r}{R}{x}{\bgamma}{S_k} \ge N] \ge p_k.
		\end{align}
		Under the isometric projection \(\pi\) for \(S_1\), we may write \(\pi S_k = \prod_{j=1}^{d-1} I^k_j\) for intervals \(I^k_j \subset \bR\) of length \(s_k\). Denote by \(m^k_j\) the midpoint of \(I^k_j\), and let \(I^k_j(L)\) and \(I^k_j(R)\) be the closures of the connected components of \(I^k_j\setminus\{m^k_j\}\). Then the collection \(\cF_k := \{\prod_{j=1}^{d-1} I^k_j(\alpha_j) \cond \alpha_j \in \{L, R\}^{d-1}\}\) forms a closed cover of \(\pi S_k\) of size \(|\cF_k| = 2^{d-1}\) which consists of cubes of side length \(\frac{s_k}{2}\). As in the proof of \Cref{thm:regularity-on-bdd+closed=compact}, we may choose \(S_{k+1} \subset S_k\) such that \(\pi S_{k+1} \in \cF_k\) and which satisfies~\eqref{eqn:many-set-crossings}:
		\begin{align*}
			\lim_{N \to \infty}\sup_{\mu \in M} \mu_n[\nscrossings{r}{R}{x}{\bgamma}{S_{k+1}} \ge N] \ge \frac{p_k}{2^{d-1}} =: p_{k+1}.
		\end{align*}
		
		By induction, we have thus constructed \(d-1\) dimensional cubes \((S_k)_{k \in \bN}\) on \(\bR^d\) with side lengths \(s_k = 2^{1-k}s_1\) which satisfy~\eqref{eqn:many-cube-crossings} for \(p_k = 2^{(1-d)k}p_0\). Note that for \(k \ge 1\) we have \(\varepsilon_k := \diam(S_k) = cs_k\) for some \(c \ge 1\) independent of \(k\), hence \(\varepsilon_k = 2^{1-k}\varepsilon_1\). With the choice \(\varepsilon = \varepsilon_k\), \eqref{eqn:nonregular-annulus} thus becomes
		\begin{align*}
			\lim_{N \to \infty} \sup_{\mu \in M} \mu[\ncrossings{\varepsilon_k}{\rho}{y}{\bgamma} \ge N] \ge \frac{p_0}{(2\varepsilon_1)^{d-1}}\varepsilon_k^{d-1}.
		\end{align*}
		In particular, we get
		\begin{align*}
			\limsup_{r \to 0} \frac{\lim_{N \to \infty} \sup_{\mu \in M} \mu[\ncrossings{r}{\rho}{y}{\bgamma} \ge N]}{r^{d-1}} \ge\;& \limsup_{k \to \infty} \frac{\lim_{N \to \infty} \sup_{\mu \in M} \mu[\ncrossings{\varepsilon_k}{\rho}{y}{\bgamma} \ge N]}{\varepsilon_k^{d-1}} \\
			\ge\;& \frac{p_0}{(2\varepsilon_1)^{d-1}} > 0,
		\end{align*}
		i.e, \(\lim_{N \to \infty} \sup_{\mu \in M} \mu[\ncrossings{r}{\rho}{y}{\bgamma} \ge N] \ne o(r^{d-1})\). The claim follows by contraposition.
	\end{proof}
	\section{Example: non-parameterized It\^o process on a Hilbert space}\label{sec:example}
	As we discussed in~\Cref{rem:tortuosity-bounds}, for not locally compact \(\cX\) it is not clear how~\Cref{thm:precompact-iff-regular} could be proved using Prokhorov-Le Cam's~\Cref{thm:Prohorov-LeCam}. Instead, the proof uses the full strength of~\Cref{thm:main-thm}, including the implication from approximate to asymtpotic tightness. This is a strong indication that~\Cref{thm:main-thm} is a useful extension of Prokhorov-Le Cam's theorem.
	
	In this section we apply~\Cref{thm:precompact-iff-regular} to a sequence of random curves arising from It\^o diffusions on a separable real Hilbert space; we refer to~\cite{da2014stochastic} for necessary background. Since infinite dimensional Hilbert spaces are not locally compact, this serves as a concrete example of a precompactness result which uses approximate tightness in an essential way.
	
	Let \(\cH\) be a separable real Hilbert space, and \(Q^{1/2}\) a symmetric Hilbert-Schmidt operator on \(\cH\). Then \(Q := (Q^{1/2})^2\) is a positive symmetric trace class operator, thus by~\cite[Proposition 4.4]{da2014stochastic} there exists a \(Q\)-Wiener process \((W_t)_{t \ge 0}\) on \(\cH\), i.e. a stochastic process having continuous trajectories started at \(W_0 = 0\) whose increments \(W_t-W_s\) are independent and normally distributed with mean \(0\) and covariance \((t-s)Q\).
	
	Let \(f : \cH\setminus\{0\}\to\cH\) be a function which is Lipschitz-continuous in \(\cH \setminus B_0(\rho)\) for every \(\rho > 0\). By Kirszbraun-Valentine theorem~\cite{kirszbraun1934zusammenziehende,valentine1945lipschitz}, for every \(\rho > 0\) there thus exists a Lipschitz-continuous map \(f_\rho : \cH \to \cH\) which agrees with \(f\) away from \(B_0(\rho)\). By~\cite[Theorem 7.2]{da2014stochastic}, the following SDE
	\begin{align*}
		\ud X^{\rho; z}_t = f_\rho(X^{\rho;z}_t) \ud t + \ud W_t, \qquad X^{\rho;z}_0 = z
	\end{align*}
	has a mild solution for every \(\rho > 0\) and \(z \in \cH\); here and elsewhere we drop the superscript \(z\) in the special case \(z = 0\) and write \(X^{\rho} := X^{\rho;0}\).
	
	Since we don not assume any regularity of \(f\) in the neighbourhood of \(0\), there is no hope for the convergence of the processes \(X^{\rho;z}\) in the limit \(\rho \to 0\) in general. However, the stopping time
	\begin{align*}
		\tau^{\rho;z} := \{\inf t \ge 0 \cond X^{\rho;z}_t \notin B_0(1)\}
	\end{align*}
	is almost surely finite for every \(\rho > 0\) and \(z \in \cH\). In particular, \(X^{\rho}|_{[0, \tau^\rho]}\) thus almost surely represents a random curve \(\gamma^\rho \in \pathset(\cH)\). Denote by \(\mu^\rho \in \cM_1(\pathset)\) the law of \(\gamma^\rho\). Under a Hilbert-Schmidt-type regularity assumption~\eqref{eqn:f-is-Hilbert-Schmidt} for \(f\) we get subsequential weak limits of the measures \(\mu^\rho\) as \(\rho \to 0\):
	\begin{prop}\label{thm:example}
		Let \(\cH\) be a separable Hilbert space, and \(f : \cH \setminus \{0\} \to \cH\) and \(\mu^\rho \in \cM_1(\pathset(\cH))\) for \(\rho > 0\) be as above. Let \((\rho_k)_{k \in \bN}\) be a decreasing sequence of positive numbers converging to \(0\). If \(\dim(\cH) < \infty\) or \(\dim(\imag(Q)) \ge 2\), then the following hold.
		\begin{enumerate}[label=(\roman*)]
			\item \label{item:example-weakly-regular} The set \(\{\mu^{\rho_k}\} \subset \cM_1(\pathset(\cH))\) is weakly regular. 
			\item \label{item:example-precomapact}If there is an orthonormal basis \(\mathcal B\) of \(\cH\) such that
			\begin{align}\label{eqn:f-is-Hilbert-Schmidt}
				\sum_{e \in \mathcal B}\sup_{x \in \annulus{\rho}{1}{0}} |\langle f(x), e\rangle|^2 < \infty, \qquad \forall \rho > 0,
			\end{align}
			then the sequence \((\mu^{\rho_k})_{k \in \bN}\) is precompact.
		\end{enumerate}
	\end{prop}
	Before proving~\Cref{thm:example}, let us make couple of remarks. In the finite dimensional case \(\cH = \bR^d\) the condition~\eqref{eqn:f-is-Hilbert-Schmidt} is always satisfied. In fact, since the closed ball \(\overline {B_0(1)} \subset \bR^d\) is compact, and the measures \(\mu^\rho\) are supported on \(\pathset(\overline {B_0(1)})\), by~\Cref{thm:precompact-iff-regular} precompactness of \((\mu^{\rho_k})_{k \in \bN}\) already follows directly from~\cref{item:example-weakly-regular}. Only in the infinite dimensional case do we make use of the assumption~\eqref{eqn:f-is-Hilbert-Schmidt}. Since the purpose of~\Cref{thm:example} is to demonstrate how~\Cref{thm:precompact-iff-regular} can be used, we do not attempt to investigate to what extent the assumptions \(\dim(\imag(Q)) \ge 2\) and~\eqref{eqn:f-is-Hilbert-Schmidt} are needed, or whether the weak limit \(\mu^\rho \xrightarrow{w} \mu\) as \(\rho \to 0\) exists.
	
	In the finite dimensional case \(\cH = \bR^d\),~\Cref{thm:example} can be used to construct examples of precompact sequences of random curves for which the annulus crossing condition~\eqref{eqn:Aizenman-Burchard-condition} from~\cite[Theorem 1.2]{aizenman1999holder} fails. Setting \(Q\) to be the identity operator, \(W\) becomes the standard Brownian motion on \(\bR^d\). Choosing \(f(x) = -\frac{(d-1)x}{2||x||^2}\), It\^o's formula shows that for every \(\rho > 0\) the process \((||X^{\rho;z}_t||)_{t \ge 0}\) coincides with the one dimensional Brownian motion up to the hitting time to \(\overline {B_0(\rho)}\). If \(T^\rho_r\) is the exit time of the process \(X^{\rho}\) from the annulus \(\annulus{r}{1}{0}\), then for \(\rho \le r < 1\) we have the Gambler's ruin probability
	\begin{align*}
		\bP\big[||X^{\rho;z}_{T^\rho_r}|| = r\big] = \frac{1-||z||}{1-r}, \qquad \forall z \in \annulus{r}{1}{0}.
	\end{align*}
	By strong Markov property, after the process \(X^\rho\) exits the ball \(B_0(R)\) for \(R \in (r, 1)\), the probability for \(X^\rho\) to exit the annulus \(\annulus{r}{1}{0}\) at \(\partial B_0(r)\) is thus \(\frac{1-R}{1-r}\). Each such event corresponds to two \(\annulus{r}{R}{0}\)-crossings made by \(\gamma^\rho\), thus for each \(N \in \bN\) we get
	\begin{align*}
		\mu^\rho[\ncrossings{r}{R}{x}{\gamma} \ge 2N] = \Big(\frac{1-R}{1-r}\Big)^N \xrightarrow{R \to 0} 1,
	\end{align*}
	where in the limit \(R \to 0\) we simultaneously decrease \(\rho\) and \(r\) so that \(\rho \le r < R\). Thus, in~\eqref{eqn:Aizenman-Burchard-condition} we are forced to choose \(\lambda_N \le 0\), which does not satisfy \(\lim_{N \to \infty}\lambda_N = \infty\) required in~\cite[Theorem 1.2]{aizenman1999holder}.
	
	\begin{proof}[Proof of~\Cref{thm:example}]
		\ref{item:example-weakly-regular} Let us first reduce weak regularity of \(\{\mu^{\rho_k}\}_{k \in \bN}\) to annulus down-crossing probability estimates. For \(x, z \in \cX, R > r > 0\), and \(n \in \bN\), define the stopping time
		\begin{align*}
			\varsigma^{\rho;z}_n = \varsigma^{\rho;z}_n(x;r,R) := \begin{cases}
				\inf\{t \ge 0 \cond \ncrossings{r}{R}{x}{X^{\rho;z}[0, t]} = n\}, & \text{if } 0 \notin \overline B_x(r),\\
				\inf\{t \ge 0 \cond \ncrossings{r}{R}{x}{X^{\rho;z}[0, t]} = n-1\}, & \text{if } 0 \in \overline B_x(r)\\
			\end{cases}
		\end{align*}
		In particular, for \(z \notin \annulus{r}{R}{x}\) the process \(X^{\rho;z}\) completes its \(n\)'th up-crossing of the annulus \(\annulus{r}{R}{x}\) at the time \(\varsigma^{\rho;z}_{2n}\), thus we have \(X^{\rho;z}_{\varsigma^{\rho;z}_{2n}} \in \partial B_x(R)\) whenever \(\varsigma^{\rho;z}_{2n} < \infty\). Since furthermore \(\varsigma^{\rho;z}_{2(n+1)} < \varsigma^{\rho;z}_{2n+1}\), we may bound
		\begin{align*}
			\bP[\varsigma^\rho_{2(n+1)} < \tau^\rho | \varsigma^\rho_{2n} < \tau^\rho] \;\le\;& \sup_{z \in \partial B_x(R)} \bP[\varsigma^\rho_{2n+1} < \tau^\rho \cond X^\rho_{\varsigma^\rho_{2n}} = z, \varsigma^\rho_{2n} < \tau^\rho]\\
			=\;& \sup_{z \in \partial B_x(R)} \bP[\varsigma^{\rho;z}_1 < \tau^{\rho;z}],
		\end{align*}
		where the last equality holds by strong Markov property. Thus, if the right hand is bounded by some \(p < 1\) independent of sufficiently small \(\rho > 0\), by the law of total probability we have
		\begin{align*}
			\mu^\rho[\ncrossings{r}{R}{0}{\gamma} \ge 2n] = \bP[\varsigma^\rho_{2n} < \tau^\rho] \le \prod_{k=0}^n \bP[\varsigma^\rho_{2(k+1)} < \tau^\rho \cond \varsigma^\rho_{2k} < \tau^\rho] \le p^n \xrightarrow{n \to \infty} 0,
		\end{align*}
		which implies regularity of \(\{\mu^{\rho_k}\}_{k \in \bN}\). We proceed to bound the probabilities \(\bP[\varsigma^{\rho;z}_1 < \tau^{\rho;z}]\).
		
		Take \(x \in \cH\) and \(R > r > 0\) such that \(R \ne ||x||\). Then we may fix \(u > 0\) small enough so that \(U := \cH \setminus (\overline {B_x(r)} \cup \overline {B_0(u)})\) is a path-connected neighbourhood of \(\partial B_x(R)\). Note that for \(\rho \le u\), the functions \(f_\rho\) and \(f_u\) coincide in \(U\). Therefore, up to the exit time of \(U\) the laws of the processes \(X^{\rho;z}\) and \(X^{u;z}\) coincide, so we may bound
		\begin{align*}
			\bP[\varsigma^{\rho;z}_1 \ge \tau^{\rho;z}] \ge \bP\big[\varsigma^{u;z}_1 \ge \tau^{u;z}, X^{u;z}[0, \tau^u] \subset U\big], \qquad \forall \rho \le u.
		\end{align*}
		A path staying in \(U\) cannot cross the annulus \(\annulus{r}{R}{x}\), thus the event \(\varsigma^{u;z}_1 \ge \tau^{u;z}\) is already contained in \(X^{u;z} [0, \tau^{u;z}] \subset U\). Regularity of \(\{\mu^{\rho_k}\}_{k \in \bN}\) thus follows if there exists \(q > 0\) such that
		\begin{align}\label{eqn:probability-bound-for-regularity}
			q_z := \bP[X^{u;z}[0, \tau^{u;z}] \subset U] \ge q, \qquad \forall z \in \partial B_x(R) \cap B_0(1).
		\end{align}
		Since \(U\) is an unbounded path connected open set, and \(X^{u;z}\) is absolutely continuous to the \(Q\)-Wiener process \(W+z\), we have \(q_z > 0\) for every \(z \in U\). In the finite dimensional case \(\cH = \bR^d\) the set \(\partial B_x(R) \subset U\) is compact, thus the infimum \(\inf_{z \in \partial B_x(R)} q_z = \min_{z \in \partial B_x(R)} q_z > 0\) is attained, proving~\eqref{eqn:probability-bound-for-regularity}.
		
		The infinite dimensional case requires more care. Since we are interested in the processes \(X^\rho_t\) only up to the stopping time \(\tau^\rho\), by Kirszbraun-Valentine theorem we may without loss of generality assume \(f(z) = 0\) whenever \(||z|| \ge 2\). Since \(f_u\) is Lipschitz, we then have \(C := \sup_{w \in \cH}||f_u(w)|| < \infty\). Then, under the event \(X^{u;z}[0, t] \cap B_x(u) = \emptyset\) we have
		\begin{align}\label{eqn:XminusWnorm}
			||X^{u;z}_t-(z + W_t)|| = \bigg|\bigg|\int_0^t f(X^{u;z}_s) \ud s\bigg|\bigg| \le \int_0^t ||f(X^{u;z}_s)|| \ud s \le Ct.
		\end{align}
		By decreasing the value of \(u\) if necessary, we may without loss of generality assume that \(U' := \overline{B_x(r+u)} \cup \overline{B_0(2u)} \subset U\) is a neighbourhood of \(\partial B_x(R)\). By~\eqref{eqn:XminusWnorm}, we then unconditionally have
		\begin{align*}
			z + W\big[0, \tfrac{u}{C}\big] \subset U' \quad \implies \quad \begin{gathered}
				\sup_{t \in [0, \frac{u}{C}]}||X^{u;z}_t|| \le \sup_{t \in [0, \frac{u}{C}]}||z + W_t|| - u,\\
				X^{u;z}\big[0,\tfrac{u}{C}\big] \subset U.
			\end{gathered}
			\begin{aligned}
				\phantom{\sup_{t \in [0, \frac{u}{C}]}}\hspace{-1em}\text{and}\\
				\ 
			\end{aligned}
		\end{align*}
		Assuming \(z \in B_0(1)\) we get \(||z + W_t||-u \ge ||W_t|| - 1 - u\). Thus, under the event \(z + W[0, \frac{u}{C}] \subset U'\) we have
		\begin{align*}
			\tau^{u;z} \le \inf\{t \ge 0 \cond ||W_t|| \ge 2+u\} =: \sigma.
		\end{align*}
		As a result, we may bound
		\begin{align*}
			q_z \ge \bP\big[(z + W[0, \tfrac{u}{C}]) \subset U', \sigma \le \tfrac{u}{C}\big], \qquad \forall z \in \partial B_x(R) \cap B_0(1).
		\end{align*}
		Suppose \(V_z \subset U'-z\) is an open neighbourhood of \(0\) such that \(V_z \cap \imag(Q)\) is unbounded and path connected. Then we may further bound
		\begin{align}\label{eqn:qz-bound}
			q_z \ge \bP[W[0,\tfrac{u}{C}] \subset V_z, \sigma \le \tfrac{u}{C}] > 0, \qquad \forall z \in \partial B_x(R) \cap B_0(1).
		\end{align}
		Our aim is to construct \(V_z\) (essentially) independently of \(z \in \partial B_x(R) \cap B_0(1)\), which then implies uniform lower bound of \(q_z\) and consequently proves regularity of \(\{\mu^{\rho_k}\}_{k \in \bN}\). Since \(\dim(\cH) = \infty\), by assumption we have \(\dim(\imag(Q)) \ge 2\). There thus exists a unit vector \(v \in \imag(Q)\) orthogonal to \(x\), i.e. \(\langle v, x\rangle = 0\). Let us check that if \(u \le r \le \frac{R}{2}\), and \(x = 0\) or \(R < \frac{||x||}{3}\), we may choose
		\begin{align*}
			V_z := \{B_0(u) + \lambda v \cond \langle z, \lambda v\rangle \ge 0\}
		\end{align*}
		Take \(\nu = w + \lambda v \in V_z\) with \(w \in B_0(u)\), and decompose \(z \in \partial B_x(R)\) as \(z = x + y\), where \(||y|| = R\). By orthogonality of \(x\) and \(v\) we have \(||x+\lambda v|| \ge ||x||\) and \(\langle y, \lambda v \rangle = \langle z, \lambda v\rangle \ge 0\). For \(x \ne 0\), we thus we get 
		\begin{align*}
			||z+\nu|| \ge\;& ||x+\lambda v|| - ||y|| - ||w|| \ge ||x|| - R - u > \frac{||x||}{3} \ge 2u, \qquad \text{and}\\
			||z+\nu-x|| \ge\;& ||y + \lambda v||-||w|| \ge \sqrt{||y||^2 + 2\lambda\langle y, v \rangle} - ||w|| \ge R-u > r+u \ge 2u,
		\end{align*}
		where the latter also holds in the case \(x = 0\). We conclude \(z + \nu \in U'\), and since \(\nu \in V_z\) was arbitrary, we get \(V_z \subset U' - z\) for every \(z \in \partial B_x(R)\). Since \(V_z\) may take one of the three possible forms (depending on whether \(\langle z, v\rangle\) is negative, positive, or zero), regularity of \(M\) follows by~\eqref{eqn:qz-bound}. 
		
		\noindent\ref{item:example-precomapact} To prove precompactness of \((\mu^{\rho_k})_{k \in \bN}\), it remains to find the sets \(Y_\varepsilon^\delta\) from~\Cref{thm:precompact-iff-regular} (since \(\cH\) is complete, we may choose \(Y_\varepsilon = \cH\)). In the finite dimensional case \(\cH = \bR^d\) we may choose \(Y_\varepsilon^\delta = \overline{B_0(1)}\). 
		
		Thus, assume \(\cH\) is infinite dimensional. Enumerate the basis as \(\mathcal B = \{e_n\}_{n \in \bN}\). For \(m \in \bN\), write \(V_m := \operatorname{span}(e_1, \ldots, e_m)\), and denote by \(\pi_m : \cH \to \cH\) the projection to its orthocomplement \(V_m^\bot := \operatorname{span}(e_{m+1}, e_{m+2}, \ldots)\). Then~\eqref{eqn:f-is-Hilbert-Schmidt} may be equivalently formulated as follows:
		\begin{align*}
			C_m(\rho) := \sup_{x \in \overline {B_0(1)} \setminus B_0(\rho)} ||\pi_m f(x)|| \xrightarrow{m \to \infty} 0, \qquad \forall \rho > 0.
		\end{align*}
		In particular, if \(X^{\rho;z}[0, T] \subset \annulus{r}{1}{0}\) for some \(r \ge \rho > 0\), then for \(t \in [0, T]\) we have
		\begin{align}\label{eqn:XtVndistance}
			\begin{split}
				\dist(X^{\rho;z}_t, V_m) = ||\pi_m X^{\rho;z}_t|| \le\;& ||\pi_m X^{\rho;z}_t - \pi_m(z + W_t)|| + ||\pi_m (z + W_t)||\\
				=\;& \bigg|\bigg|\int_0^t \pi_m f_r(X^{\rho;z}_s)\ud s\bigg|\bigg| + ||\pi_m z + \pi_m W_t||\\
				\le\;& C_m(r)T + ||z|| + ||\pi_m W_t||.
			\end{split}
		\end{align}
		Using the reverse triangle inequality instead and bounding \(||X^\rho_0|| \le 1\), the above computation in the case \(m = 0\) yields
		\begin{align}\label{eqn:Xt-bound}
			||X^{\rho;z}_t|| \ge ||W_t|| - C_0(r)t - 1.
		\end{align}
		We will use~\eqref{eqn:Xt-bound} to bound the time it takes the process \(X^{\rho;z}\) to exit the annulus \(\annulus{r}{1}{0}\)
		and~\eqref{eqn:XtVndistance} to bound how far the process \(X^{\rho;z}\) will get from the subspace \(V_m\) before it exists \(\annulus{r}{1}{0}\).
		
		Fix \(\varepsilon, \delta > 0\), and let \(R = \frac{\delta}{3}\). By the regularity proof above, there exists \(r \in (0, R), u > 0\), and \(N \in \bN\) such that
		\begin{align*}
			\bP[\varsigma^\rho_{2N}(0;r,R) < \tau^\rho] < \frac{\varepsilon}{3}, \qquad \forall \rho < u.
		\end{align*} 
		Denote by
		\begin{align*}
			T_n := \inf\{t \ge 0 \cond X^\rho[\varsigma^\rho_{2n-1}, \varsigma^\rho_{2n-1}+t] \not\subset \annulus{r}{1}{0}\}
		\end{align*}
		the time it takes for the process \(X^\rho\) to exit \(\annulus{r}{1}{0}\) after the \(n\)'th upcrossing of \(\annulus{r}{R}{0}\). By~\eqref{eqn:Xt-bound}, for any \(t \in [\varsigma^\rho_{2n-1}, \varsigma^\rho_{2n-1}+T_n-1]\) we then must have \(||W_{t+1} - W_t|| < 2+C_0(r)\). By strong Markov property of the process \(X^\rho\), for \(T \in \bN\) we thus get
		\begin{align*}
			\bP\big[T_n > T\big] \le\;& \bP\bigg[\bigcap_{k=1}^{T} ||W_{\varsigma^\rho_{2n-1}+k}-W_{\varsigma^\rho_{2n-1}+k-1}|| < 2+C_0 \bigg] \\
			=\;& \Big(\underbrace{\bP\big[||W_1|| > 2+C_0\big]}_{<1}\Big)^{T}.
		\end{align*}
		There thus exists \(T \ge 0\) such that
		\begin{align*}
			\bP\big[T_n > T\big] < \frac{\varepsilon}{3N}, \qquad \forall n \in \bN.
		\end{align*}
		Fix \(m \in \bN\) such that
		\begin{gather*}
			\bP\Big[\sup_{t \in [0, T]} ||\pi_m W_t|| \ge R\Big] < \frac{\varepsilon}{3N}, \quad \text{and} \quad
			C_m(r) < \frac{R}{T}.
		\end{gather*}
		On the one hand, we have \(X^\rho[\varsigma^\rho_{2n}, \varsigma^\rho_{2n+1}] \subset \overline{B_0(r)}\). On the other hand, if we take \(t \in [\varsigma^\rho_{2n-1}, \varsigma^\rho_{2n-1}+T_n]\), then by~\eqref{eqn:XtVndistance} we have
		\begin{align*}
			||\pi_m X^\rho_t|| \le \frac{T_n}{T}R + R + \sup_{s \in [0, T_n]}||\pi_m \tilde W^n_s||.
		\end{align*}
		where \(\tilde W^n_t := W_{\varsigma^\rho_{2n-1} + t} - W_{\varsigma^\rho_{2n-1}}\). Thus, if we have \(\tau^\rho \le \varsigma^\rho_{2N}\), and \(\{T_n \le T\}\) for every \(1 \le n \le N\), then we have
		\begin{align*}
			\sup_{t \in [0, \tau^\rho]} ||\pi_m X^\rho_t|| \le 2R + \max_{1 \le n \le N} \sup_{s \in [0, T]} ||\tilde W^n_s||.
		\end{align*}
		Under the event \(\sup_{t \in [0, \tau^\rho]} ||\pi_m X^\rho_t|| \ge \delta = 3R\), this implies \(\sup_{s \in [0, T]} ||\pi_m \tilde W^n_s|| \ge R\) for some \(1 \le n \le N\). Choosing \(Y_\varepsilon^\delta := \overline{B_0(1)} \cap V_m\) we have \(\dist(X^\rho_t, Y_\varepsilon^\delta) = ||\pi_m X^\rho_t||\) for every \(t \in [0, \tau^\rho]\), therefore for \(\rho \in (0, u)\) we get
		\begin{align*}
			\mu^\rho\Big(\pathset \setminus \pathset\big(B_{Y_\varepsilon^\delta}(\delta)\big)\Big) \;&=\bP\Big[\sup_{t \in [0, \tau^\rho]} ||\pi_m X^\rho_t|| \ge \delta\Big] \\
			&\le \bP[\varsigma^\rho_{2N} < \tau^\rho] + \sum_{n=1}^N \bigg(\bP\big[T_n > T\big] + \bP\Big[\sup_{s \in [0, T]} ||\pi_m \tilde W^n_s || > R\Big]\bigg)\\
			&< \frac{\varepsilon}{3} + N\Big(\frac{\varepsilon}{3N} + \frac{\varepsilon}{3N}\Big) = \varepsilon.
		\end{align*}
		Note that \(Y^\delta_\varepsilon\) is a bounded closed subset of the finite dimensional vector space \(V_m\), thus it is compact. By~\Cref{thm:precompact-iff-regular}, we conclude precompactness of \(M\).
	\end{proof}
	
	\section{Conclusions}
	In this paper we refined the Prokhorov-Le Cam's theorem and provided a short and elementary proof which, contrast to standard textbook proofs, does not rely on outer measures or Riesz-Markov-Kakutani representation theorem from functional analysis. We also provided a characterization of precompact collections of random curve collections on locally path connected metric spaces, and refined the characterizing condition further in Euclidean (sub)spaces. We also provided examples which demonstrate novelty of both the refined Prokhorov-Le Cam's theorem, and the precompactness condition for a sequence of random curves.
	\appendix
	\crefalias{section}{appendix}
	\section{Conditioning on closed separable sets preserves precompactness}\label{sec:appendix}
	In this appendix we prove that conditioning on closed separable sets preserves precompactness of measures used in the proof of equivalence of approximate and asymptotic tightness. Recall the notation \(\conditioned{\mu}{A} \in \cM_1(\cX)\) for the measure \(\mu \in \cM_1(\cX)\) conditioned on a Borel set \(A \subset \cX\).
	\conditioning*
	\Cref{thm:conditioning-preserves-precompactness} follows directly from the version where, instead of conditioning the measures, we restrict them (\Cref{thm:restricting-preserves-precompactnes}). Denote by \(\cM_{\le 1}(\cX)\) the collection of Borel measures \(\mu\) on \(\cX\) with total mass \(\mu(\cX) \le 1\). For \(\mu \in \cM_{\le 1}(\cX)\) and a Borel subset \(A \subset \cX\), denote by \(\restricted{\mu}{A}\) the restricted measure
	\begin{align}
		\restricted{\mu}{A}(B) := \mu(B \cap A) \qquad \forall B \subset \cX \text{ Borel}.
	\end{align}
	We equip \(M_{\le 1}(\cX)\) with the topology of weak convergence.
	\begin{prop}\label{thm:restricting-preserves-precompactnes}
		Let \(\cX\) be a metric space. Suppose \(\mathcal C\) is a precompact subset of \(\cM_{\le 1}(\cX)\) and \(Y \subset \cX\) is a separable closed subset. Then the restricted measures
		\begin{align*}
			\restricted{\mathcal C}{Y} := \{\restricted{\mu}{Y} \cond \mu \in \mathcal C\}
		\end{align*}
		form a precompact subset of \(\cM_{\le 1}(\cX)\).
	\end{prop} 
	Let us first show how~\Cref{thm:conditioning-preserves-precompactness} follows from~\Cref{thm:restricting-preserves-precompactnes} and the following version of Portmanteau's theorem for bounded measures. Recall that a Borel set \(B \subset \cX\) is called a continuity set of a measure \(\mu \in \cM_{\le 1}(\cX)\) if \(\mu(\partial B) = 0\).
	\begin{theorem}[Portmanteau's theorem; {combination of \cite[Theorem 13.16]{klenke2013probability} and \cite[Theorem 2.2]{billingsley1999convergence}}]\label{thm:Portmanteau}
		Let \(\cX\) be a metrizable topological space, and let \(\mu, \mu_1, \mu_2, \mu_3, \ldots\) be measures in \(\cM_{\le 1}(\cX)\) such that \(\lim_{n \to \infty}\mu_n(\cX) = \mu(\cX)\). Then the following are equivalent.
		\begin{enumerate}[label=(\roman*)]
			\item\label{item:weak-convergence} The sequence \((\mu_n)_{n \in \bN}\) converges weakly to \(\mu\).
			\item\label{item:closed-sets} \(\limsup_{n \to \infty} \mu_n(F) \le \mu(F)\) for every closed \(F \subset \cX\).
			\item\label{item:open-sets} \(\liminf_{n \to \infty}\mu_n(U) \ge \mu(U)\) for every open \(U \subset \cX\).
			\item\label{item:continuity-sets} \(\lim_{n \to \infty} \mu_n(B) = \mu(B)\) for every continuity set \(B \subset \cX\) of \(\mu\).
			\item\label{item:pi-continuity-sets} There exists a \(\pi\)-system \(\Pi\) such that every open set \(U \subset \cX\) is a countable union of sets in \(\Pi\), and every \(A \in \Pi\) satisfies \(\lim_{n \to \infty}\mu_n(A) = \mu(A)\).
		\end{enumerate}
	\end{theorem}
	\begin{proof}[Proof of~\Cref{thm:conditioning-preserves-precompactness}]
		Suppose \(\mathcal C \subset \cM_1(\cX)\) and \(Y \subset \cX\) satisfy the assumptions in~\Cref{thm:conditioning-preserves-precompactness}. By~\Cref{thm:restricting-preserves-precompactnes}, the restricted measures \(\restricted{\mathcal C}{Y}\) form a precompact subset of \(\cM_{\le 1}(\cX)\). Every sequence \((\mu_n)_{n \in \bN}\) in \(\mathcal C\) thus has a subsequence \((n_j)_{j \in \bN}\) along which the weak limit \(\restricted{\mu_{n_j}}{Y} \xrightarrow{w} \nu\) exists. From weak convergence we get
		\begin{align*}
			\nu(\cX) = \lim_{j \to \infty} \restricted{\mu_{n_j}}{Y}(\cX) = \lim_{j \to \infty}\mu_{n_j}(Y) \ge \inf_{\mu \in \mathcal C}\mu(Y) > 0,
		\end{align*}
		while Portmanteau (\Cref{thm:Portmanteau} \ref{item:weak-convergence}\(\implies\)\ref{item:open-sets}) implies
		\begin{align*}
			\nu(\cX\setminus Y) \le \liminf_{j \to \infty}\restricted{\mu_{n_j}}{Y}(\cX\setminus Y) = 0,
		\end{align*}
		thus we conclude \(\nu(Y) = \nu(\cX) = \lim_{j \to \infty} \mu_{n_j}(Y) > 0\). In particular, for any closed subset \(F \subset \cX\) we get
		\begin{align*}
			\limsup_{j \to \infty}\conditioned{(\mu_{n_j})}{Y}(F) = \frac{\limsup_{j \to \infty} \restricted{\mu_{n_j}}{Y}(F \cap Y)}{\lim_{j \to \infty} \mu_{n_j}(Y)} \ge \frac{\nu(F \cap Y)}{\nu(Y)} = \conditioned{\nu}{Y}(F),
		\end{align*}
		where the inequality follows from Portmanteu (\Cref{thm:Portmanteau} \ref{item:weak-convergence} \(\implies\) \ref{item:closed-sets}) applied to the sequence \(\restricted{\mu_{n_j}}{Y} \xrightarrow{w} \nu\). By applying the converse of this implication we conclude the weak convergence  \(\conditioned{(\mu_{n_j})}{Y} \xrightarrow{w} \conditioned{\nu}{Y}\). As \(\conditioned{\nu}{Y} \in \cM_1(\cX)\) and the sequence \((\conditioned{(\mu_n)}{Y})_{n \in \bN}\) was arbitrary, we conclude precompactness of \(\conditioned{\mathcal C}{Y}\). 
	\end{proof}
	In the proof of~\Cref{thm:restricting-preserves-precompactnes} we will construct the limiting measure using the following version of Carath\'eodory's extension theorem. Let \(\setring\) be a ring of sets (i.e. collection of sets closed under finite unions and set differences). A function \(\mu : \setring \to [0,\infty]\) is said to be a premeasure on \(\setring\) if for any sequence \((A_k)_{k \in \bN}\) of pairwise disjoint sets in \(\setring\) such that \(A := \bigcup_{k \in \bN} A_k \in \setring\) we have \(\mu(A) = \sum_{k \in \bN}\mu(A_k)\). We denote by \(\sigma(\setring)\) the sigma-algebra generated by \(\setring\), which is the smallest sigma-algebra containing every element of \(\setring\).
	\begin{theorem}[Carath\'eodory's extension theorem, {\cite[Theorem 1.41]{klenke2013probability}}]\label{thm:Caratheodory-extension}
		A finite pre-measure on a ring of sets \(\setring\) extends uniquely to a measure on \(\sigma(\setring)\).
	\end{theorem}
	\begin{proof}[Proof of~\Cref{thm:restricting-preserves-precompactnes}]
		Take a sequence \((\mu_n)_{n \in \bN}\) in \(\mathcal C\); by precompactness we may pass to a subsequence to assume without loss of generality that the weak limit \(\mu_n \xrightarrow{w} \mu \in \cM_{\le 1}(\cX)\) exists. Our aim is to show that along a further subsequence \((n_j)_{j \in \bN}\) the weak limit \(\mu_{n_j}|_Y \xrightarrow{w} \nu\) exists.
		
		Let \(S\) be a countable dense subset of \(Y\). Since \(\mu\) is a finite measure, there is a countable dense subset \(\tilde Q\) of \((0, \infty)\) such that
		\begin{align}\label{eqn:continuity-sets-of-mu}
			\mu(\partial B_s(r)) = 0 \qquad \forall s \in S, r \in \tilde Q.
		\end{align}
		Let \(\setring\) be the ring of sets generated by \(\mathcal U := \{B_s(r) \cond s \in S, r \in \tilde Q\}\), and for \(A \in \setring\) write \(A_\circ := A \cap Y^\circ\) and \(A_\partial := A \cap \partial Y\). Taking finite unions and differences of the sets in \(\mathcal U\) and applying~\eqref{eqn:continuity-sets-of-mu} shows that
		\begin{align}\label{eqn:continuity-sets-of-mu-1}
			\mu(\partial A) = 0 \qquad \forall A \in \setring.
		\end{align}
		Since \(\setring\) is countable, there exists a countable dense subset \(Q\) of \((0, \infty)\) such that (write \(\setring_\partial := \{A_\partial \cond A \in \setring\}\))
		\begin{align}\label{eqn:continuity-sets-of-mu-2}
			\mu(\partial B_{I}(r)) = 0 \qquad \forall r \in Q, I \in \setring_\partial.
		\end{align}
		By diagonal extraction over the countable set \(Q \times \setring_\partial\), there exists a subsequence \((n_j)_{j \in \bN}\) such that the following limits exist:
		\begin{align*}
			\tilde \nu_r(I) := \lim_{j \to \infty}\mu_{n_j}(B_I(r) \cap Y) \qquad \forall r \in Q, I \in \setring_\partial.
		\end{align*}
		Note that \(\tilde\nu_r(I)\) is monotone in \(r \in Q\), so it admits the limit
		\begin{align*}
			\tilde\nu(I) := \lim_{r \to 0}\tilde\nu_r(I).
		\end{align*}
		Our goal is to prove the following claims:
		\begin{enumerate}[label=(\roman*)]
			\item\label{item:premeasure} \(\tilde \nu\) extends uniquely to a Borel measure on \(\partial Y\).
			\item\label{item:weak-limit}The Borel measure \(\nu\) on \(Y\) defined by
			\begin{align}
				\nu(B) := \mu(B \setminus \partial Y) + \tilde \nu(B \cap \partial Y) \qquad \forall B \subset Y \text{ Borel}
			\end{align}
			is the weak limit of \((\mu_{n_j}|_Y)_{j \in \bN}\). 
		\end{enumerate}
		To establish the above claims we will use the following auxiliary result. Let \((I_n)_{n \in \bN}\) be a sequence of pairwise disjoint sets in \(\setring_\partial\), and write \(I := \bigcup_{n \in \bN} I\). Suppose \((n^I_j)_{j \in \bN}\) is a subsequence of \((n_j)_{j \in \bN}\) such that the following limits exist:
		\begin{align}\label{eqn:ttilde-definition}
			\ttilde \nu_r(I) := \lim_{j \to \infty}\mu_{n^I_j}(B_I(r) \cap Y), \qquad \ttilde \nu(I) := \lim_{r \to 0}\ttilde \nu_r(I);
		\end{align}
		such subsequences exist by diagonal extraction, while the limit \(\ttilde\nu(I)\) exists by monotonicity of \(r \mapsto \ttilde\nu_r(I)\). Denote by \(I^{\partial\circ}\) the interior of \(I \subset \partial Y\) with respect to the subspace topology on \(\partial Y\). We claim that the following holds:
		\begin{align}\label{eqn:ttilde-equality}
			\ttilde \nu(I) = \sum_{k \in \bN} \tilde \nu(I_k), \qquad \text{whenever } \mu(\overline I \setminus I^{\partial \circ}) = 0,
		\end{align}
		so in particular the value of \(\ttilde \nu(I)\) does not depend on the choice of the subsequence \((n^I_j)_{j \in \bN}\). We will prove~\eqref{eqn:ttilde-equality} after showing how the claims~\ref{item:premeasure} and~\ref{item:weak-limit} follow from it.
		
		\noindent \textbf{Proof of \ref{item:premeasure}.} Let \((I_k)_{k \in \bN}\) be a sequence of pairwise disjoint sets in \(\setring_\partial\), and suppose \(I := \bigcup_{k \in \bN}I_k \in \setring_\partial\). By definition, there then exists \(A \in \setring\) such that \(I = A_\partial\), so we get
		\begin{align*}
			\mu(\overline I\setminus I^{\partial \circ}) \le \mu(\overline A \setminus A^\circ) = \mu(\partial A) = 0,
		\end{align*}
		where the last equality is~\eqref{eqn:continuity-sets-of-mu}. We may thus apply~\eqref{eqn:ttilde-equality} to get
		\begin{align*}
			\tilde \nu(I) = \ttilde \nu(I) = \sum_{k \in \bN} \tilde \nu(I_k),
		\end{align*}
		proving that \(\tilde \nu\) is a premeasure on \(\setring_\partial\). By~\Cref{thm:Caratheodory-extension}, \(\tilde \nu\) extends uniquely to a measure on \(\sigma(\setring_\partial)\), which coincides with the Borel sigma-algebra for \(\partial Y\).
		
		\noindent\textbf{Proof of~\ref{item:weak-limit}.} Note that every open subset of \(Y\) can be expressed in terms of countable unions of sets in the \(\pi\)-system \(\{A \cap Y \cond A \in \setring\}\).  By Portmanteu's~\Cref{thm:Portmanteau} (\ref{item:pi-continuity-sets} \(\implies\) \ref{item:weak-convergence}), to prove \(\mu_{n_j}|_Y \xrightarrow{w} \nu\) it thus suffices to check that
		\begin{align}
			\lim_{j \to \infty} \mu_{n_j}(A \cap Y) = \nu(A \cap Y) \qquad \forall A \in \setring. \label{eqn:nu-continuity-sets} 
		\end{align}
		Take \(A \in \setring\). For any \(r \in Q\) we have
		\begin{align}\label{eqn:decomposition}
			\mu_{n_j}(A \cap Y) =\;& \mu_{n_j}(B_{A_\partial}(r) \cap Y) + \mu_{n_j}(A_\circ \setminus B_{A_\partial}(r)) - \mu_{n_j}(Y \cap (B_{A_\partial}(r) \setminus A)).
		\end{align}
		Note that \(\partial(A_\circ \setminus B_{A_\partial}(r)) \subset \partial A \cup \partial B_{A_\partial}(r)\), so by ~\eqref{eqn:continuity-sets-of-mu-1} and~\eqref{eqn:continuity-sets-of-mu-2} we conclude that \(A_\circ \setminus B_{A_\partial}(r)\) is a continuity set of \(\mu\). Applying the definition of \(\tilde \nu_r\) to the first term and Portmanteau (\Cref{thm:Portmanteau}, \ref{item:weak-convergence} \(\implies\) \ref{item:continuity-sets}) to the second term on the right hand side of~\eqref{eqn:decomposition} thus yields
		\begin{align*}
			\lim_{j \to \infty}\Big(\mu_{n_j}(B_{A_\partial}(r) \cap Y) + \mu_{n_j}(A_\circ \setminus B_{A_\partial}(r))\Big) =\;& \tilde \nu_r(A_\partial) + \mu(A_\circ \setminus B_{A_\partial}(r))\\
			\xrightarrow{r \to 0}\;& \tilde \nu(A_\partial) + \mu(A_\circ) = \nu(A \cap Y).
		\end{align*}
		Applying Portmanteau (\Cref{thm:Portmanteau} \ref{item:weak-convergence} \(\implies\) \ref{item:closed-sets}) to the closed set \(\overline {B_A(r)} \setminus A^\circ\) containing \(Y \cap (B_{A_\partial}(r) \setminus A)\) we can bound the last term in~\eqref{eqn:decomposition} as follows:
		\begin{align*}
			\limsup_{j \to \infty }\mu_{n_j}(Y \cap (B_{A_\partial}(r) \setminus A)) \le \mu(\overline B_A(r) \setminus A^\circ) \xrightarrow{r \to 0} \mu(\partial A) = 0;
		\end{align*}
		the last equality holds by~\eqref{eqn:continuity-sets-of-mu-1}. Hence, taking first \(j \to \infty\) and then \(r \to 0\) on the right hand side of~\eqref{eqn:decomposition} yields~\eqref{eqn:nu-continuity-sets} for every \(A \in \setring\). It thus remains to prove~\eqref{eqn:ttilde-equality}.

		\noindent\textbf{Proof of~\eqref{eqn:ttilde-equality}.} A key observation is that for every \(A \in \setring\) and \(\varepsilon > 0\) there exist \(A^{+\varepsilon}, A^{-\varepsilon} \in \setring\) satisfying
		\begin{align}\label{eqn:A-regularity}
			B_{A^{-\varepsilon}}(\varepsilon) \subset A \subset B_{A^{-\varepsilon}}(2\varepsilon) \qquad \text{and} \qquad B_{A}(\varepsilon) \subset A^{+\varepsilon} \subset B_{A}(2\varepsilon)
		\end{align}
		Indeed, if \(A = B_s(r) \in \mathcal U\), we may choose \(A^{\pm\varepsilon} = B_s(r^\pm)\), where \(r^- \in (r-2\varepsilon, r-\varepsilon) \cap \tilde Q\) and \(r^+ \in (r+\varepsilon, r+2\varepsilon) \cap \tilde Q\). If \(A, \tilde A \in \setring\) both admit sets \(A^{\pm\varepsilon}, \tilde A^{\pm\varepsilon} \in \setring\) satisfying~\eqref{eqn:A-regularity}, then it is straightforward to check that we may choose
		\begin{align*}
			(A \cup \tilde A)^{\pm\varepsilon} := A^{\pm\varepsilon} \cup \tilde A^{\pm\varepsilon}, \qquad \text{and} \qquad (A\setminus \tilde A)^{\pm\varepsilon} := A^{\pm\varepsilon}\setminus \tilde A^{\mp\varepsilon}.
		\end{align*}
		Since every element of \(\setring\) is obtained from the sets in \(\mathcal U\) via finite unions and set differences, the existence of the sets \(A^{\pm \varepsilon} \in \setring\) for any \(A \in \setring\) satisfying \eqref{eqn:A-regularity} follows by induction.
		
		Take a sequence \((I_k)_{k \in \bN}\) of pairwise disjoint sets in \(\setring_\partial\), and write \(I := \bigcup_{k \in \bN} I_k\), and let \((n^I_j)_{j \in \bN}\) be a subsequence of \((n_j)_{j \in \bN}\) such that the limits in~\eqref{eqn:ttilde-definition} exist. By definition, \(I_k = (A_k)_\partial =: A_{k\partial}\) for some \(A_k \in \setring\); since \(\setring\) is a ring, we may without loss of generality assume that \(A_k\) are pairwise disjoint. Fix \(r \in Q\), \(\varepsilon \in (0, r) \cap Q\) and \(\delta \in (0,\varepsilon) \cap Q\). Note that  we have \(B_{A^{-\varepsilon}_{k\partial}}(\delta) \subset A_k\), while \((A_k)_{k \in \bN}\) are pairwise disjoint. Furthermore, since \(A^{-\varepsilon}_{k\partial} \subset I\) and \(\delta < r\), we have \(B_{A^{-\varepsilon}_{k\partial}}(\delta) \subset B_I(r)\). By monotonicity and countable additivity of the measure \(\mu_{n_j}\) we thus get
		\begin{align*}
			0 \le \mu_{n^I_j}(B_I(r) \cap Y) - \sum_{k \in \bN}\mu_{n^I_j}(B_{A^{-\varepsilon}_{k\partial}}(\delta) \cap Y) \le \mu_{n^I_j}\Big(\overline{B_I(r)} \setminus \bigcup_{k \in \bN}B_{A^{-\varepsilon}_{k\partial}}(\delta)\Big).
		\end{align*}
		Taking \(\limsup_{j \to \infty}\) above and applying Fatou's lemma for the sum and Portmanteau (\Cref{thm:Portmanteau} \ref{item:weak-convergence} \(\implies\) \ref{item:closed-sets}) to the right hand side yields
		\begin{align*}
			0 \le \ttilde \nu_r(I) - \sum_{k \in \bN}\tilde \nu_{\delta}(A^{-\varepsilon}_{k\partial}) \le \mu\Big(\overline{B_I(r)} \setminus \bigcup_{k \in \bN}B_{A^{-\varepsilon}_{k\partial}}(\delta)\Big).
		\end{align*}
		Note that \(\tilde \nu_\delta(A^{-\varepsilon}_{k\partial})\) is dominated by \(\tilde \nu_{\varepsilon}(A^{-\varepsilon}_{k\partial})\), while \(\sum_{k \in \bN}\tilde \nu_{\varepsilon}(A^{-\varepsilon}_{k\partial}) \le \mu(\cX) < \infty\). Hence, we may take the limit \(\delta \to 0\) and apply dominated convergence theorem for the sum and reverse monotone convergence theorem to the right hand side to get
		\begin{align*}
			0 \le \ttilde \nu_r(I) - \sum_{k \in \bN}\tilde \nu(A^{-\varepsilon}_{k\partial}) \le \mu\big(\overline {B_I(r)} \setminus \bigcup_{k \in \bN}A^{-\varepsilon}_{k\partial}\Big).
		\end{align*}
		Since \(\tilde \nu(A^{-\varepsilon}_{k\partial})\) is increasing with decreasing \(\varepsilon\) and \(\mu\) is a finite measure, we may take the limits \(\varepsilon \to 0\) and \(r \to 0\) in this order and apply monotone convergence theorem to the sum and reverse monotone convergence to the right hand side to get (write \(A^\circ_{k\partial} := A^\circ_k \cap \partial Y\))
		\begin{align*}
			0 \le \ttilde \nu(I) - \sum_{k \in \bN}\lim_{\varepsilon \to 0}\tilde\nu(A^{-\varepsilon}_{k\partial}) \le \mu(\overline I \setminus \bigcup_{k \in \bN}A^\circ_{k\partial}).
		\end{align*}
		We have \(I^{\partial\circ} \subset I = \bigcup_{k \in \bN} A_{k\partial}\). Since furthermore \(A_{k\partial}\setminus A^\circ_{k\partial} \subset \partial A_k\) is a continuity set of \(\mu\) for every \(k \in \bN\) (by~\eqref{eqn:continuity-sets-of-mu-1}), we get
		\begin{align*}
			\mu(\overline I\setminus \bigcup_{k \in \bN}A^\circ_{k\partial}) = \mu(\overline I \setminus I^{\partial\circ}) + \mu\Big(\bigcup_{k \in \bN}A_{k\partial} \setminus A^\circ_{k\partial}\Big) = \mu(\overline I \setminus I^{\partial\circ}).
		\end{align*}
		We arrive at the following estimate:
		\begin{align}\label{eqn:ttilde-estimate}
			0 \le \ttilde \nu(I) - \sum_{k \in \bN}\lim_{\varepsilon \to 0}\tilde \nu(A^{-\varepsilon}_{k\partial}) \le \mu(\overline I \setminus I^{\partial\circ}).
		\end{align}
		The above estimate holds for any sequence \((I_k)_{k \in \bN}\) of pairwise disjoint sets in \(\setring_\partial\). Specializing to the case where \(I_1 = A_\partial\) for some \(A \in \setring\) and \(I_m = \emptyset\) for every \(m \ge 2\) yields
		\begin{align*}
			0 \le \ttilde \nu(A_\partial) - \lim_{\varepsilon \to 0}\tilde\nu(A^{-\varepsilon}_{\partial}) \le \mu(\overline {A_\partial} \setminus A_\partial^{\partial\circ}) \le \mu(\partial A) = 0,
		\end{align*}
		where the last equality holds by~\eqref{eqn:continuity-sets-of-mu-1}. We thus conclude \(\lim_{\varepsilon \to 0}\tilde\nu(A^{-\varepsilon}_{\partial}) = \tilde \nu(A_\partial)\) for every \(A \in \setring\). Substituting this to~\eqref{eqn:ttilde-estimate} and assuming \(\mu(\overline I \setminus I^{\partial \circ}) = 0\) yields~\eqref{eqn:ttilde-equality}, finishing the proof.
	\end{proof}
	
	\bibliography{precbib}
	\bibliographystyle{alpha}
\end{document}